\documentclass[11pt,a4paper,reqno]{article}
\usepackage{amssymb,amsthm,amsmath,graphics,epsfig,mathrsfs}
\usepackage{hyperref}
\usepackage{authblk}
\usepackage{bbm}
\usepackage{enumitem}
\usepackage{fullpage}
\usepackage{verbatim}

\usepackage{pgf,tikz}
\usetikzlibrary{decorations.pathreplacing}
\usetikzlibrary{patterns}

\usepackage[numeric,initials,nobysame,msc-links,abbrev]{amsrefs}
\renewcommand{\eprint}[1]{\href{https://arxiv.org/abs/#1}{arXiv:#1}}
\renewcommand{\MR}[1]{}
\newcommand{\pageafter}[1]{#1~pp.}
\BibSpec{article}{%
+{} {\PrintAuthors} {author}
+{,} { \textit} {title}
+{.} { } {part}
+{:} { \textit} {subtitle}
+{,} { \PrintContributions} {contribution}
+{.} { \PrintPartials} {partial}
+{,} { } {journal}
+{} { \textbf} {volume}
+{} { \PrintDatePV} {date}
+{,} { \issuetext} {number}
+{,} { \pageafter} {pages}
+{,} { } {status}
+{,} { \PrintDOI} {doi}
+{,} { available at \eprint} {eprint}
+{} { \parenthesize} {language}
+{} { \PrintTranslation} {translation}
+{;} { \PrintReprint} {reprint}
+{.} { } {note}
+{.} {} {transition}
+{} {\SentenceSpace \PrintReviews} {review}
}
\BibSpec{collection.article}{%
+{} {\PrintAuthors} {author}
+{,} { \textit} {title}
+{.} { } {part}
+{:} { \textit} {subtitle}
+{,} { \PrintContributions} {contribution}
+{,} { \PrintConference} {conference}
+{} {\PrintBook} {book}
+{,} { } {booktitle}
+{,} { \PrintDateB} {date}
+{,} { \pageafter} {pages}
+{,} { } {status}
+{,} { \PrintDOI} {doi}
+{,} { available at \eprint} {eprint}
+{} { \parenthesize} {language}
+{} { \PrintTranslation} {translation}
+{;} { \PrintReprint} {reprint}
+{.} { } {note}
+{.} {} {transition}
+{} {\SentenceSpace \PrintReviews} {review}
}

\newtheorem{Theorem}{Theorem}[section]

\newtheorem{Lemma}[Theorem]{Lemma}
\newtheorem{Proposition}[Theorem]{Proposition}
\newtheorem{Corollary}[Theorem]{Corollary}
\newtheorem{Remark}[Theorem]{Remark}

\newtheorem{Claim}[Theorem]{Claim}
\newtheorem{Definition}[Theorem]{Definition}

\newtheorem{Assumption}[Theorem]{Assumption}

\setlist[itemize]{leftmargin=*}
\setlist[enumerate]{leftmargin=*,label=(\roman*),ref=(\roman*)}

\newcommand{\be}{\begin{equation}}
\newcommand{\en}{\end{equation}}

\newcommand{\cA}{\ensuremath{\mathcal A}}
\newcommand{\cB}{\ensuremath{\mathcal B}}
\newcommand{\cC}{\ensuremath{\mathcal C}}
\newcommand{\cD}{\ensuremath{\mathcal D}}
\newcommand{\cE}{\ensuremath{\mathcal E}}

\newcommand{\cG}{\ensuremath{\mathcal G}}
\newcommand{\cH}{\ensuremath{\mathcal H}}

\newcommand{\cL}{\ensuremath{\mathcal L}}
\newcommand{\cM}{\ensuremath{\mathcal M}}
\newcommand{\cN}{\ensuremath{\mathcal N}}

\newcommand{\cP}{\ensuremath{\mathcal P}}

\newcommand{\cS}{\ensuremath{\mathcal S}}
\newcommand{\cT}{\ensuremath{\mathcal T}}

\newcommand{\cV}{\ensuremath{\mathcal V}}
\newcommand{\cW}{\ensuremath{\mathcal W}}
\newcommand{\cX}{\ensuremath{\mathcal X}}

\newcommand{\bbE}{{\ensuremath{\mathbb E}} }

\newcommand{\bbG}{{\ensuremath{\mathbb G}} }

\newcommand{\bbM}{{\ensuremath{\mathbb M}} }
\newcommand{\bbN}{{\ensuremath{\mathbb N}} }

\newcommand{\bbP}{{\ensuremath{\mathbb P}} }
\newcommand{\bbQ}{{\ensuremath{\mathbb Q}} }
\newcommand{\bbR}{{\ensuremath{\mathbb R}} }
\newcommand{\bbS}{{\ensuremath{\mathbb S}} }

\newcommand{\bbW}{{\ensuremath{\mathbb W}} }
\newcommand{\bbX}{{\ensuremath{\mathbb X}} }
\newcommand{\bbY}{{\ensuremath{\mathbb Y}} }
\newcommand{\bbZ}{{\ensuremath{\mathbb Z}} }
\newcommand{\md}{\mathrm{d}}

\newcommand{\1}{\mathbbm{1}}
\DeclareMathOperator{\supp}{supp}

\let\a=\alpha    \let\d=\delta  \let\e=\varepsilon
 \let\g=\gamma     \let\k=\kappa  \let\l=\lambda
      \let\o=\omega      
  \let\s=\sigma \let\t=\tau   
\let\y=\upsilon  \let\z=\zeta
\let\D=\Delta   \let\G=\Gamma  \let\L=\Lambda 
\let\O=\Omega      
\newcommand{\pc}{p_{\mathrm{c}}}

\title{Crossings and diffusion in Poisson driven marked random connection models}
\author[1]{Alessandra Faggionato}
\author[2]{Ivailo Hartarsky}
\affil[1]{\small Department of Mathematics, Sapienza University of Rome. P.le Aldo Moro 2, 00185 Roma, Italy, \texttt{faggiona@mat.uniroma1.it}}
\affil[2]{\small Universit\'e Claude Bernard Lyon 1, CNRS, Centrale Lyon, INSA Lyon, Universit\'e Jean Monnet, ICJ UMR5208, 69622 Villeurbanne, France, \texttt{hartarsky@math.univ-lyon1.fr}}
\date{\today}

\begin{document}

\maketitle

\begin{abstract}
We first study crossing statistics in random connection models (RCM) built on marked Poisson point processes on $\mathbb R^d$. Under general assumptions, we show  exponential tail bounds for the number of crossings of a box contained in the infinite cluster for supercritical intensity of the point process, and  percolation in slabs, in analogy with the Grimmett–-Marstrand theorem. We then present several applications to transport and diffusion phenomena. In particular, we prove  the non-degeneracy of the effective homogenized matrix arising in the large-scale limit of random walks, exclusion processes, and resistor networks on the RCM, and  the non-degeneracy of the effective diffusion constant for one-dimensional diffusion operators on the Euclidean graph associated with the RCM.
 As examples, we apply our results to Poisson--Boolean models and Mott variable range hopping random resistor network, providing a fundamental ingredient used in the derivation of Mott's law.
\end{abstract}

\noindent\textbf{MSC2020:} 60K35; 60K37; 60D05; 82B21
\\
\textbf{Keywords:} Random connection model; 
effective homogenized matrix; Poisson--Boolean model; Mott variable range hopping; percolation

\section{Introduction}
\subsection{Crossings and transport}
Random (possibly weighted)
undirected graphs find numerous modelling applications. A remarkable one is  the analysis of transport in disordered media, modelled e.g.\ by considering  a resistor network, a flow network, a random walk, an interacting particle system, a 1d diffusion, etc.\ on the graph. In all these models edge weights, viewed as conductances, tune the intensity of local flow  along the edges (unweighted graphs can be thought with unit conductances). If transport  takes place at macroscopic scale, the medium represented by the graph is a conductor. Naturally, transport requires the presence of enough crossings (roughly, macroscopic paths) in the graph.

We deal with media which are  microscopically inhomogeneous but spatially homogeneous, from a statistical viewpoint. Thus, we are mainly interested in random weighted graphs embedded in $\bbR^d$ whose law is stationary w.r.t.\ Euclidean translations. Further assuming ergodicity w.r.t.\ Euclidean translations and, possibly, some irreducibility given by the connectivity of the graph, one can often express the large-scale conduction property of the medium by means of the so-called effective homogenized matrix $D$. Several derivations of such stochastic homogenization and  scaling limits rely (among other things) on the crossing properties of the graph,  which also guarantee that $D$ is non degenerate (see e.g.\ \cite{Berger07} and \cite{Mathieu07} based also on \cite{Barlow04}, \cite{Zhikov06}*{Section~7}, \cite{Armstrong18}*{Section 2}, \cite{Armstrong23}*{Section 2}).
Other derivations (with applications to random resistor networks and interacting particle systems) have been developed in \cites{Faggionato25,Faggionato22,Faggionato23a,ChiariniInpreparation}
to derive qualitative homogenization and scaling limits for a very large class of random graphs and without relying on crossing properties (thus ensuring universality but allowing the resulting $D$ to be degenerate).

Let us say that a random graph has a good crossing statistics if with high probability, for some $\k>0$, the subgraph given by the edges with conductances larger than $\k$ has at least order $\ell^{d-1}$ vertex-disjoint   crossings  of a  box of side length $\ell$ in any  given direction.  Then, roughly, good crossing statistics ensure the non-degeneracy of $D$ 
in the 2-scale convergence approach of  \cites{Zhikov06,Faggionato23a}  (at the basis of  \cites{Faggionato25,Faggionato22,ChiariniInpreparation}). For the supercritical Bernoulli percolation on $\bbZ^d$ (with unit conductances) the good crossing statistics were first proved by Kesten \cite{Kesten82}*{Chapter~11} and Chayes--Chayes \cite{Chayes86}.  Kesten's proof is largely based on \cite{Grimmett84} (cf.~\cite{Kesten82}*{Remark~(ii), p.~339}). As presented in Section~\ref{sec_prob_bounds} the proof in \cite{Chayes86} can be easily extended  to get  the  same result  for the supercritical Bernoulli percolation cluster itself. In particular, being a connected graph, the supercritical Bernoulli percolation cluster with unit or, more generally, positively lower bounded conductances is an effective graph to model transport in disordered conductors.

In applications, graphs without infinite clusters do not allow macroscopic transport, whereas infinite clusters themselves are commonly used as models for disordered media. When studying amorphous materials without any underlying lattice structure, the widely used supercritical Bernoulli percolation cluster on $\mathbb{Z}^d$ is no longer a natural model. Hence, in this work, we aim to establish the aforementioned crossing property for more general random infinite clusters. 

\subsection{Main results}
We consider a random connection model (RCM)  on a marked homogeneous Poisson point process (PPP)  with density $\rho$ on $\bbR^d$, with i.i.d.\ marks,  whose connection function $\varphi$, specifying the probability to have an edge between sites $x\not=y$, depends on $y-x$ and on the marks of $x$ and $y$ (see Section~\ref{sec_model_results} for a formal definition). When the connection function is an indicator (as in the Poisson--Boolean model), the insertion of the edge is deterministic and fully determined by the marked PPP.  We point out that the RCM on a marked PPP is also widely used in the context of  scale-free graphs (see e.g.~\cite{Gracar22}).

Under suitable assumptions on $\varphi$ (cf.~Assumptions~\ref{assumere}), we prove the following.  Assume that, for some $\lambda>0$, the RCM has a unique infinite cluster for densities $\rho>\lambda$. Then (cf.~Theorem~\ref{teo_crossings}), for all $\rho>\l$ there exist positive constants $c,c'$ such that  for all $\ell \geq 1$ the unique infinite cluster of the RCM has at least $c \ell^{d-1}$ left-right (LR) vertex-disjoint crossings of the box $\L_\ell:=[-\ell,\ell]^d$ with probability at least $1-e^{-c' \ell^{d-1}}$. Moreover, (cf.~Theorem~
\ref{th:slab}), under the same conditions, a sufficiently thick two-dimensional slab contains an infinite cluster, as in the classical result of Grimmett--Marstrand for Bernoulli percolation \cite{Grimmett90}.  In many applications, it is known that $\lambda$ can be taken equal to the critical density of the model. Moreover, in \cite{Chebunin25}, the uniqueness of the infinite cluster is reduced to checking a simple irreducibility condition. We also point out that Theorem \ref{teo_crossings} extends to a much  larger class the results of \cite{Faggionato21} obtained for RCM on marked PPP  whose connection function is an indicator function satisfying somewhat restrictive monotonicity properties.

By Menger's theorem, the maximal number of vertex-disjoint LR crossings corresponds to the minimal cardinality of vertex cutsets separating the two faces of the  box $\L_\ell$ under consideration (strictly speaking, the two half-strips, see Definition~\ref{def_LR_crossing} below). The supercritical percolation cluster can be thought of as a capacity network where each edge has capacity one  and Theorem~\ref{teo_crossings} could be thought as a first step in investigating  capacity networks built on marked RCMs. For both classical and  more recent results about capacity networks we refer e.g.\ to \cites{Grimmett84,Chayes86,Rossignol18,Zhang18} and references therein.

From Theorem~\ref{teo_crossings}, we derive the non-degeneracy of the effective homogenized matrix $D$ for several transport models built on the infinite cluster of the RCM with positively lower bounded conductances. For these results we refer to Theorem~\ref{teo2} for random resistor networks and to Proposition~\ref{teo3} for 1d diffusions along the edges of the graph. We point out that the matrix $D$ for random resistor networks is the same appearing in the large scale limit of the random walk and (for hydrodynamics and fluctuations) of the symmetric simple exclusion process on the graph, cf.~\cites{Faggionato23a,Faggionato22,ChiariniInpreparation}. Hence Theorem~\ref{teo2}  naturally completes the analysis in~\cites{Faggionato23a,Faggionato25,Faggionato22,ChiariniInpreparation} for the RCM on marked PPP. In Section \ref{sec_gen_boolean}, we treat two important examples: the Poisson--Boolean model (cf.~Theorem~\ref{teo_booleano}) and the graph associated to Mott variable range hopping (v.r.h.) with cutoff (cf.~Theorems~\ref{teo_mott_vrh} and \ref{teo_mott_vrh_2}). For the latter, our results on the LR crossings are fundamental  in   \cite{Faggionato26} to complete the proof of the physics Mott's law for the anomalous decay of the conductivity in doped semiconductors and other amorphous solids, in the physically relevant case of  marks with different sign w.r.t.\ the Fermi energy set equal to zero in Section~\ref{sec:Mott:vrh}  (the crossing analysis for a fixed sign was already provided by \cite{Faggionato21}). In particular, our contribution here is fundamental for \cite{Faggionato26}*{Theorem~7.11 and Corollary~7.12}.

\subsection{Proof outline and comments on the strategy}
\label{subsec:outline}
 We conclude with a rough outline of the proof strategy of the crossing result given by Theorem~\ref{teo_crossings}, emphasising some technical novelties and difficulties. We point out that this proof is the most demanding one  and Theorem~\ref{th:slab}  emerges as a simple corollary of an intermediate result in this proof, while Theorems~\ref{teo2},~\ref{teo_booleano},~\ref{teo_mott_vrh},~\ref{teo_mott_vrh_2} and Proposition~\ref{teo3} follow relatively simply from Theorem~\ref{teo_crossings} and known results, so we direct the reader to the respective full proofs.

Firstly, by a classical argument going back to  \cites{Chayes86,Grimmett84}, it is enough to prove good crossing statistics in a quasi-two-dimensional setting. Then, LR crossings are constructed by studying a suitable growth process (see Section~\ref{sec_tanemura}). It is related to the one introduced in \cite{Tanemura93}*{Section~4} for the blob model (i.e.\ the Poisson--Boolean model with a deterministic radius). However, we were led to extend and adapt Tanemura's process, in order to recover crossings contained in the infinite cluster.
 
Comparison to the growth process is obtained via a new dynamical renormalization scheme. In \cite{Faggionato21}, the RCM was discretized and then a renormalization scheme inspired by the classical Grimmett--Marstrand \cite{Grimmett90} approach for Bernoulli percolation was applied. In the present work, we reason more directly, bypassing discretization and using the `seedless' renormalisation of \cite{Duminil-Copin21}. However, the marks of the RCM induce a disparity between different points of the PPP. This prevents uniformly lower bounding the probability of connecting a generic marked point to infinity (in particular, the validity of condition~(a) in \cite{Duminil-Copin21}*{Theorem~3} is no longer guaranteed).
Therefore, in Section~\ref{sec_connections}, we are led to enhance the argument of \cite{Duminil-Copin21} with a new method in order to avoid `bad' marks.

The above main difficulty of the models under consideration represented by `bad' marks also impacts the dynamical renormalisation performed in Section~\ref{proof_teo_crossings} and illustrated in Figure~\ref{fig:renorm}. Indeed, the necessity to find `good' marks progressively offsets renormalisation, imposing correcting the positioning at each step. Furthermore, preserving sufficient independence throughout the dynamical exploration in the continuum setting also requires some care in devising the renormalization.

Finally, let us comment on an alternative route in light of recent progress. Quantitative proofs of the Grimmett–Marstrand slab result have been obtained for Bernoulli percolation on $\bbZ^d$ \cite{Duminil-Copin21} and on general transitive graphs with polynomial growth \cite{Contreras24} (also see \cite{Easo23} for faster growth), using dynamical and static renormalization, respectively. Our approach follows \cite{Duminil-Copin21}, using dynamical renormalization and an exploration algorithm adapted from Tanemura \cite{Tanemura93} to establish stochastic domination between the explored region and the corresponding region in the supercritical Bernoulli percolation.
Alternatively, one could have tried to extend the methods of \cite{Contreras24} to RCMs. This approach is technically more involved but would yield 
additional information. In this case, the renormalized process would have finite-range dependence and, by the Liggett--Schonmann--Stacey theorem \cite{Liggett97}, 
would stochastically dominate a Bernoulli percolation that can be taken to be supercritical. By controlling local uniqueness events (see   Definition~\ref{def:local:uniqueness}) at the renormalisation scale, disjoint crossings in the renomalized process would lead to  disjoint crossings in the original model.  In particular, under this alternative approach, Tanemura’s construction would not be required. However, the difficulties arising from points with bad marks would remain.

\section{Model and main results}\label{sec_model_results}

\subsection{Random connection model crossings}Let $\bbM$ be a  Polish space (i.e.~a complete separable metric space) and let $\nu$ be a probability measure on $\bbM$ endowed with the $\s$--algebra of Borel subsets. We call $\bbM$  \emph{mark space} and $\nu$ \emph{mark distribution}.

We fix once and for all a dimension $d\geq 2$ and  denote by $\cL$ the Lebesgue measure on $\bbR^d$. 
The product space $\bbR^d\times\bbM$ is a measure space endowed with the  measure  $\cL\otimes \nu$ on the Borel $\s$--algebra. We denote by $\pi: \bbR^d\times\bbM \to \bbR^d$ the canonical projection, i.e.~$\pi(x,m)=x$ for all $(x,m)\in \bbR^d\times\bbM$.

We fix a \emph{connection function}, i.e.\ a measurable function $\varphi: (\bbR^d \times \bbM) ^2 \to [0,1]$
which is symmetric,  i.e.~$\varphi (t,t')=\varphi(t',t)$ for all $t,t'\in \bbR^d\times\bbM$.

We introduce the set $\O:=\{\o \in \bbR^d\times\bbM\,:\, \o\text{ if locally finite}\}$. The  $\s$--field of measurable subsets of $\O$  is the one generated by  $\{\o \in \O\,:\, \sharp (\o \cap A) =n\}$, where $A$ varies among the Borel subsets of $\bbR^d\times\bbM$ and $n$ varies in $\bbN$.   Given $\rho>0$ we denote by $P_{\rho,\nu}$ the  law on $\O$ of the PPP  on $\bbR^d\times \bbM$ with intensity measure $\rho \cL \otimes \nu$. The measure $P_{\rho,\nu}$ coincides also with the law of the marked PPP $\tilde \xi$ on $\bbR^d$ obtained as follows. First take a realization $\xi$ of a  homogeneous PPP on $\bbR^d$ with intensity $\rho$, afterwards mark all points in $\xi$ by i.i.d.\ random variables with value in $\bbM$   and distribution $\nu$ independently from the rest, finally consider the set $\tilde \xi :=\{(x,m_x)\in\bbR^d\times\bbM: x\in \xi\}$ where $m_x$ is the mark of the point $x$. Then  $\tilde \xi$ has law $P_{\rho,\nu}$.

We point out  that $P_{\rho,\nu}$ is concentrated on the set  $\O_*$ given by configurations $\o$ such that  $\hat \o:=\pi(\o)$ is a locally finite subset of $\bbR^d$ and $t=t'$ whenever $t,t'\in \o$ and $\pi(t)=\pi(t')$. A configuration $\o\in \O_*$ can be written as $\{ (x,m_x): x\in \hat \o \}$, $m_x$ denoting the mark of point $x$.  Points in $\hat\o$ can be enumerated in  a measurable way \cite{Last18}*{Proposition~6.2}. 
When referring to the above enumeration we write $\hat\o=\{x_1,x_2,\dots\}$ and we denote by  $m_i$ the  mark  of $x_i$  (i.e.\ $m_i:=m_{x_i}$)

We let $J:=\{(i,j)\in \bbN_+\times\bbN_+\,:\, 1\leq i <j\}$, where $\bbN_+:=\{1,2,\dots\}$, and denote by  $\underline{u}=(u_{i,j})_{(i,j)\in J}$ a generic element on $ [0,1]^J$. We endow $ [0,1]^J$ with the product Lebesgue probability measure  $Q=\bigotimes _{(i,j)\in J} \md x $.

\begin{Definition}[Graph $G_d$] \label{def_G}
Given $(\o,\underline u)\in \O_*\times [0,1]^J$  we denote by $G_d(\o,\underline{u})$ the undirected graph on $\bbR^d$ with vertex set $\hat\o=\{x_1,x_2,\dots\}$ and edges $\{x_i,x_j\}$ where  $1\leq i<j$ and $u_{i,j} \leq \varphi \big( (x_i, m_i) ) , (x_j,m_j) )$.
The graph $G_d(\o,\underline{u})$ defined a.s.\ on the space $\O\times[0,1]^J$ with probability $P_{\rho,\nu}\otimes Q$  is called  Poisson driven marked random connection model (RCM)  with intensity $\rho$,  mark distribution $\nu$ and connection function $\varphi$. We denote this random graph by  \emph{RCM}$(\rho,\nu,\varphi)$.
\end{Definition}

Briefly, the RCM$(\rho,\nu,\varphi)$ is obtained as follows: sample a homogeneous PPP $\xi$ on $\bbR^d$ with intensity $\rho$, then mark  points $x$ in $\xi$  by i.i.d.\ random variables $m_x$ with value  in $\bbM$ and common distribution $\nu$ independently from the rest, then insert an edge between $x\not =y$ in $\xi$ with probability $\varphi \big( (x,m_x), (y,m_y)\big)$ independently for any pair $\{x,y\}$ in $\xi$ and independently from the rest.

The subindex $d$ is used to stress that the graph $G_d$ lives in $\bbR^d$ and not in $\bbR^d\times \bbM$. Although this is natural for applications, in the mathematical analysis in the next sections it will be more natural to work with graphs in $\bbR^d\times \bbM$, for which we reserve the notation $G$. 

Given $\ell>0$, we consider the strip
 $\bbS_{\ell}:=\bbR\times [-\ell,\ell ]^{d-1}$ and we write   it as  the disjoint union 
$ \bbS_{\ell}^-\cup \L_{\ell} \cup \bbS_{\ell}^+$, 
where 
\[\bbS_{\ell}^-:=  \{x\in \bbS_{\ell}\,:\, x_1<  -{\ell} \}\,,\qquad \bbS_{\ell}^+:=  \{x\in \bbS_{\ell}\,:\, x_1> {\ell}\}\,, \qquad
 \L_{\ell}:=[-{\ell},{\ell}]^d\,.\]

\begin{Definition}[LR crossing] \label{def_LR_crossing}  Let $(\o,\underline u)\in \O_*\times [0,1]^J$.
Given $\ell>0$ a left-right (LR) crossing  of  the box $\L_{\ell}$ in the graph  $G_d(\o,\underline u)$    is any  sequence of distinct points $y_1$, $y_2$,$\dots$, $y_n \in \hat \o $  with $n\geq 3$  such that 
\begin{itemize}
\item $\{y_i ,y_{i+1}\}$ is an edge of $G_d(\o,\underline u)$  for all $i=1,2,\dots, n-1$;
\item $y_1 \in \bbS_{\ell}^{-} $ and   $y_n \in \bbS_{\ell}^+ $;
\item $y_2, y_3,\dots, y_{n-1}\in \L_{\ell}$.
\end{itemize}
\end{Definition}

\begin{Assumption}\label{assumere} We assume the following for a given triple $(\l,\nu,\varphi)$ of intensity $\lambda>0$, mark measure $\nu$ and connection function $\varphi$:
\begin{enumerate}
\item\label{ass:stationary} $\varphi$ is  stationary, 
 i.e.~$\varphi\big((x,m), (x',m')\big)= \varphi \big((0,m), (x'-x,m') \big)$ for all $x,x'\in \bbR^d$, $m,m'\in\bbM$;
\item\label{ass:finite:support} $\varphi$ has finite spatial support, i.e.~there exists ${\ell}_*>0$ such that   $\varphi \big((x,m), (x',m')\big)= 0$ whenever  $|x-x'|_\infty \geq {\ell}_*$;
\item\label{ass:symmetry} $\varphi$ fulfils  the following symmetries: $\varphi ( (0, m), (x,m'))= \varphi ((0,m), (y,m'))$ for any $x,y\in \bbR^d$ and $m,m'\in \bbM$ such that $y$ is obtained from $x$ by a permutation of the coordinates of $x$ or $y$ is obtained from $x$ by flipping the sign  of  one coordinate of $x$;
\item\label{ass:irreducible} for all $\rho\geq \l$ the  graph $G_d$ has at most one  unbounded connected component $P_{\rho,\nu}\otimes Q$--a.s.;
\item\label{ass:supercritical} the  graph $G_d$ has an unbounded connected component $P_{\l,\nu}\otimes Q$--a.s.
\end{enumerate}
\end{Assumption}

Let us comment on Assumptions~\ref{assumere}.

As we will see, Assumption~\ref{assumere}\ref{ass:finite:support} can sometimes be relaxed (see e.g.\ Section~\ref{subsec:boolean}), but working under sharp conditions on the tail of $\varphi$ is delicate in general. 
Assumption~\ref{assumere}\ref{ass:finite:support} implies that  
\[\int_{\bbR^d} \md x \int_\bbM \nu(\md m) \int_\bbM\nu(\md m')  \varphi\big( (0,m), (x,m') \big) <\infty\,.\]
In particular, we are in the setting of \cite{Chebunin25}*{Section~4}.

The symmetry Assumption~\ref{assumere}\ref{ass:symmetry} is used only for the so called ``square root trick" in Eq.~\eqref{eq:sqrt:trick} below. It is a bit stronger than what the proof requires. We note that, while working with less symmetry is sometimes possible (see e.g.\ \cite{Bezuidenhout94}), it would complicate proofs significantly, so we prefer to work under this simplifying assumption.

Assumption~\ref{assumere}\ref{ass:irreducible} is known to hold in several models. Chebunin and Last \cite{Chebunin25} (see there for a detailed discussion on previous results) provide a very general criterion implying Assumption~\ref{assumere}\ref{ass:irreducible}. To state it we introduce,  for $n\in\bbN_+$,  the $n$--th convolution of $\varphi$  given by the function $\varphi ^{(n)} : (\bbR^d\times\bbM)^2\to [0,1]$ defined recursively by
\begin{equation}
\label{eq:def:phin}
 \varphi^{(1)} (t_1,t_2):= \varphi (t_1,t_2) \qquad \text{ and }\qquad
\varphi^{(n+1)} (t_1,t_2):= \int \varphi^{(n)} (t_1,t) \varphi(t, t_2) (\l \cL\otimes\nu) (dt) \,.
\end{equation}
Using Assumptions~\ref{assumere}\ref{ass:stationary} and~\ref{assumere}\ref{ass:finite:support} together with  \cite{Chebunin25}*{Remark~5.2 and Theorem~12.1},
Assumption~\ref{assumere}\ref{ass:irreducible} is satisfied if $\varphi$ is irreducible for $P_{\l,\nu}$, i.e. 
\be\label{eq:irreducible:basic}\sup_{n\geq 1} \varphi^{(n)} (t,t')>0 
\text{ for 
$(\lambda\cL\otimes\nu)^{\otimes 2}$--a.a.~$(t,t')\in (\bbR^d\times\bbM)^2$}\,.
\en
By \cite{Chebunin25}*{Remark~5.2}, in \eqref{eq:def:phin} and \eqref{eq:irreducible:basic}, we can drop $\l$, dealing just with $\cL\otimes\nu$.
By \cite{Chebunin25}*{Proposition~5.1(i),(iv) and Eq.~(5.1)} the above criterion \eqref{eq:irreducible:basic} has the following equivalent formulation:
\begin{equation}
\label{eq:irreducible}\bbP\big( t\text{ is connected to }t' \text{ in } G(\xi\cup \{t,t'\}) \big)>0  \text{ for $(\lambda\cL\otimes\nu)^{\otimes 2}$--a.a.\ $(t,t')\in (\bbR^d\times\bbM)^2$}\,,
\end{equation}
where $\xi$ is a PPP on  $\bbR^d\times \bbM$ with 
law $P_{\l,\nu}$ and
$ G(\xi\cup \{t,t'\}) $
is the  random graph with vertex set  $\xi\cup \{t,t'\}$  obtained by putting an edge between two distinct vertices $t_1,t_2$ with probability $\varphi(t_1,t_2)$, independently for any pair of vertices and from the rest (we will come back to the RCM  $G$ in Section~\ref{sec_preliminaries}). 
By \cite{Chebunin25}*{Proposition~5.5} any unmarked (i.e.\ with Dirac $\nu$) RCM satisfies Assumption~\ref{assumere}\ref{ass:irreducible} for any $\l$.

Finally, we observe that  Assumptions~\ref{assumere}\ref{ass:irreducible} and~\ref{assumere}\ref{ass:supercritical} together with stochastic domination imply  for any $\rho\geq \l$ that the graph $G_d$ has a unique unbounded connected component (i.e.~infinite cluster) $P_{\rho,\nu}$--a.s.

\smallskip

We can now state our first main result.
\begin{Theorem}[Lower bounds on LR crossings] \label{teo_crossings} Suppose the triple $(\l,\nu,\varphi)$ satisfies  Assumptions~\ref{assumere}. Then for any $\rho>\l$ there exist $c_1,c_2>0$ such that, for all $\ell$ large enough, we have
\be\label{ananas}
P_{\rho,\nu}\otimes Q \left( \cN_{\ell}   \geq  c_1 {\ell}^{d-1}\right)\geq 1- \exp\left(- c_2 \,{\ell} ^{d-1}\right)\,,
\en
where $\cN_{\ell}$ denotes the maximal number of vertex-disjoint LR crossings of $\L_{\ell}$ included in the unique  infinite cluster of $G_d$.
\end{Theorem}
The proof of Theorem \ref{teo_crossings} is given in Section~\ref{proof_teo_crossings},  where we mainly combined and applied what developed in  Sections~\ref{sec_preliminaries}, \ref{sec_connections} and  \ref{sec_tanemura}.

\begin{Remark}\label{docile}
We call length of a LR crossing the number of its edges. By a large deviation estimate for  Poisson random variables, for $\ell\geq 1$ the event  $\sharp(\hat\o\cap\L_{\ell})\leq 2\rho (2\ell)^d$ has $P_{\rho,\nu}$--probability at least $1- e^{-c {\ell}^d}$ for some $c>0$. When the above event takes place,  there can be at most $2\rho (2{\ell})^d/\left(C {\ell}-1\right)$ vertex-disjoint LR crossings of $\L_{\ell}$ with length at least $C {\ell}$, for $C>1$. By taking $C$ large to assure $2^{d+1}\rho/C<c_1/3$ with $c_1$ as in Theorem~\ref{teo_crossings}, at cost to update $c_1$ and $c_2$ in \eqref{ananas}, we get that Theorem~\ref{teo_crossings} remains valid if we define $\cN_{\ell}$ as  the maximal number of vertex-disjoint LR crossings of $\L_{\ell}$  with length upper bounded by $C{\ell}$ and  included in the unique  infinite cluster of $G_d$. We stress that, by Assumption~\ref{assumere}\ref{ass:finite:support}, any LR crossing has length at least $2{\ell}/{\ell}_*$.
\end{Remark}

The proof of Theorem~\ref{teo_crossings} has the following notable byproduct (which is an immediate consequence of the Borel-Cantelli lemma and Eq.~\eqref{2d_bound} in Section~\ref{proof_teo_crossings}).
\begin{Theorem}[Percolation in slabs]
    \label{th:slab} Suppose the triple $(\lambda,\nu,\varphi)$ satisfies Assumptions~\ref{assumere}. Then for any $\rho>\lambda$ there exists $L>0$ such that the slab $\Sigma_L=\bbR^2\times[-L,L]^{d-2}$ satisfies
    \[P_{\rho,\nu}\otimes Q(\Sigma_L\text{ contains an unbounded path of $G_d$})=1.\]
\end{Theorem}

Stating Theorem~\ref{th:slab} explicitly in the present version was prompted by the posterior but independent work of Penrose \cite{Penrose25}, where Theorem~\ref{th:slab} is proved in the particular case $\varphi:((x,m),(x',m'))\mapsto\phi(|x-x'|_p)$ for some $p\in[1,\infty]$ and nonincreasing $\phi:[0,\infty)\to[0,1]$ with finite support.\footnote{The version in \cite{Penrose25} is slightly different, stating that,  for a  thick enough slab, with positive probability, the origin (added to the PPP with density $\rho$ and marked independently with law $\nu$) has an infinite cluster in the RCM. By using Campbell's identity (cf.\ \cite{Daley88}), one can easily show the equivalence of the two formulations.} The proof of \cite{Penrose25} employs the original Grimmett--Marstrand approach, contrary to the present work.
Previously, Last, Penrose and Zuyev \cite{Last17a} proved Theorem~\ref{th:slab} for the Poisson--Boolean model with bounded random radii, that is $\bbM=[0,1]$ and $\varphi((x,m),(x',m'))=\1_{|x-x'|_2\le m+m'}$. The result was also proved by Dembin and Tassion \cite{Dembin26} for unbounded radii under minimal moment assumptions (cf.\ Section~\ref{subsec:boolean}).

\subsection{Transport via resistor networks}\label{trans_RN}
We now move to transport in the infinite cluster of our RCM. 
A fundamental modelling for transport is provided by resistor networks. The  directional  conductivity measures  the electrical transport  under the effect of a constant  (in space and time) electric field.
Let us first recall  these concepts (for a more details the reader can see~\cite{Faggionato25}).

Consider a weighted undirected graph (also called \emph{conductance model}) $\cG:=(\cV,\cE,c)$, where $\cV$ and $\cE$ denote as usual the vertex set and the edge set, respectively, while the function  $c:\cE\to[0,\infty)$ assigns to each edge a non-negative weight. We take  $\cE\subset\{ \{x,y\}\,:\,x\not=y\text{ in }\cV\}$ and  $\cV\in \cN(\bbR^d)$, $\cN(\bbR^d)$ being the space of locally finite subsets of $\bbR^d$.

\begin{Definition}[Resistor network ${\rm RN}_{\ell}$] Given $\ell>0$  we associate to $\cG$ the resistor network ${\rm RN}_{\ell}$  with nodes the points in $\cV\cap S_{\ell}$ and with electric filaments  given by the edges  $\{x,y\}\in \cE$  included in the strip $S_{\ell}$ and intersecting $\L_{\ell}$, the edge $\{x,y\}$ having conductance $c_{x,y}:=c(\{x,y\})$.
\end{Definition} 

Having the resistor network $\text{RN}_{\ell}$, we consider the electrical potential $v_{\ell}:\cV\cap S_{\ell}\to\bbR$ with value $1$ on $S_{\ell}^-$ and $0$ on $S_{\ell}^+$.
The potential $v_{\ell}$ is related to the current $\iota (x,y)$ that flows along the electric filament from $x$ to $y$ by the identity $\iota(x,y)=c_{x,y}( v_{\ell}(x) -v_{\ell}(y))$.

\begin{Definition}[Directional conductivity $\s_{\ell}$] Given ${\ell}>0$ we associate to $\cG$  the directional conductivity $\s_{\ell}$   defined as the effective conductivity of the resistor network  ${\rm RN}_{\ell}$ along the first direction $e_1$ under the above  electric potential $v_{\ell}$, i.e.\ $ \s_{\ell}: =  \sum _{x\in \cV\cap S_{\ell}^-}
   \; \sum_{\substack{y \in  \cV \cap \L_{\ell}:\\y\sim x}} \iota _{x,y}$.
\end{Definition}

We  recall  Rayleigh’s monotonicity law for resistor networks  (see e.g.~\cite{Doyle84}):
if $(\cV,\cE)$ contains the graph $(\cV',\cE')$ and the conductance functions satisfy $c_{x,y}\geq c'_{x,y}$ for all $\{x,y\}\in\cE'$, then $\s_{\ell} (\cV,\cE,c)\geq\s_{\ell}(\cV',\cE',c')$

As a crucial application of the above monotonicity, explaining also the link with LR crossings, we have the following. We include the standard proof for completeness (see e.g.\ \cites{Kesten82,Chayes86,Faggionato23a}).
\begin{Lemma}[Lower bound on $\s_{\ell}$] \label{lemma_LB}  
Let $\ell\geq 1$.  Suppose that the weighted undirected graph $\cG=(\cV,\cE,1)$ contains $N$ vertex-disjoint LR crossings\footnote{The definition of LR crossings for a generic graph is similar to  Definition~\ref{def_LR_crossing}.} of $\L_{\ell}$.   Then 
$\s_{\ell}  \geq  \frac{N^2}{2\sharp(\cV\cap \L_{\ell})}$.  
\end{Lemma}

\begin{proof}
The effective resistance of the $i$-th LR crossing given by $l_i$ filaments in series equals $l_i$ and therefore its effective conductivity is $1/l_i$. Since the vertex-disjoint LR crossings behave as electric chains in parallel,  their total  effective conductivity is  $\sum_{i=1}^{N} 1/l_i$. By Jensen's inequality $\sum_{i=1}^{N}\frac{1}{l_i}\geq  \frac{N^2}{\sum_{i=1}^{N}l_i}$, while the  last denominator is bounded by $N+ \sharp(\cV\cap\L_{\ell})\leq 2 \sharp(\cV\cap \L_{\ell})$. Finally,  Rayleigh's monotonicity principle allows to conclude.
\end{proof}

\medskip
We can finally present our second main result, concerning transport in the infinite cluster of our RCM. 
More precisely, we suppose that the triple $(\l,\nu,\varphi)$ satisfies Assumptions~\ref{assumere}.  Then, given  $\rho\geq\l$, the graph $G_d(\o,\underline u)$ has a unique infinite cluster $\cC(\o,\underline u)$ for  $P_{\rho,\nu}\otimes Q$--a.a.~$(\o,\underline u)$. 
We are interested in transport in the weighted undirected  graph
$\left(\cC(\o,\underline u), \cE(\o,\underline u), 1\right)$, 
where the set $\cE(\o,\underline u)$ is given by  the edges in $G_d(\o,\underline u)$ between two points in the infinite cluster $\cC(\o,\underline u)$.

The large scale homogenization of the medium is encoded in the so-called  \emph{effective homogenized matrix} $D(\rho)$. 
The definition of $D(\rho)$ follows from the general setting presented in \cite{Faggionato25} (for the reader's convenience we recall the definition  in Appendix~\ref{app_fine}).
 Here we just mention that $D(\rho)$ is a non--negative definite symmetric $d\times d$--matrix and that, 
thanks to Assumption~\ref{assumere}\ref{ass:symmetry},   $D(\rho)$ is proportional to the identity, namely
\be\label{eq_diagonale}
D(\rho)=\k(\rho)\mathbb{I}\,.
\en

We recall  that $D(\rho)$ describes  the scaling limit of the direction conductivity  of the infinite cluster with unit conductances.  Indeed we can apply  \cite{Faggionato25}*{Corollary~2.7} to this specific case  (see Appendix~\ref{app_fine} for some comments) to obtain the  following.
 
 \begin{Proposition} [\cite{Faggionato25}*{Corollary~2.7}] \label{prop_scaling} 
Assume the triple $(\l,\nu,\varphi)$ satisfies Assumptions~\ref{assumere} and take $\rho\geq \l$. Then 
  we have
 \be
 \lim_{\ell \to +\infty}{(2\ell)}^{2-d}\s_{\ell}(\o, \underline u)=\rho\, \k(\rho)\qquad P_{\rho,\nu}\otimes Q\text{--a.s.}
 \en
 where $\s_{\ell}(\o,\underline u)$ refers to the weighted graph $\left(\cC(\o,\underline u), \cE(\o,\underline u), 1\right)$.
 \end{Proposition}
 On the other hand,
the results in \cite{Faggionato25}  do not imply that $\k(\rho)>0$.
  As a combination of Theorem~\ref{teo_crossings},   Lemma~\ref{lemma_LB}  and Proposition~\ref{prop_scaling} it is now immediate to  get our second main result.
 \begin{Theorem}[Non-degeneracy of $D(\rho)$, $\k(\rho)$] \label{teo2} Suppose that the triple $(\l,\nu,\varphi)$ satisfies  Assumptions~\ref{assumere}. Then for any $\rho>\l$  the  effective homogenized matrix $D(\rho)$ is strictly positive definite, i.e.~$\k(\rho)>0$.
 \end{Theorem}
\begin{proof}
 By Theorem \ref{teo_crossings} and the Borel-Cantelli Lemma, for $P_{\rho,\nu}\otimes Q$--a.a.~$(\o,\underline u)$, for $\ell$ large enough  the cluster $\cC(\o,\underline u )$ has at least $c_1 {\ell}^{d-1}$ vertex-disjoint LR crossings of $\L_{\ell}$. By Lemma~\ref{lemma_LB}  this implies, for $P_{\rho,\nu}\otimes Q$--a.a.~$(\o,\underline u)$, that $\s_{\ell}(\o,\underline u)\geq c_1 \frac{{\ell}^{2d-2}}{2\sharp(\hat \o \cap \L_{\ell})}$ for $\ell$ large. To control the denominator we observe that  the  ergodicity of the homogeneous PPP on $\bbR^d$ implies that $\lim_{\ell\to+\infty}\frac{\sharp(\hat \o \cap \L_{\ell})}{(2{\ell})^d}=\rho$ 
$P_{\rho,\nu}\otimes Q$--a.s. Hence, we conclude that $\liminf_{\ell \to +\infty}{\ell}^{2-d}\s_{\ell} (\o, \underline u)>0$ $P_{\rho,\nu}\otimes Q$--a.s. By combining this with  Proposition~\ref{prop_scaling}  we get  that $\k(\rho)>0$ and therefore $D(\rho)$ is strictly positive definite.
\end{proof}

We also point the reader to Appendix~\ref{transport_diff} for another application of Theorem~\ref{teo_crossings} to diffusions on the system of straight line segments corresponding to the RCM, in the spirit of \cite{Zhikov06}.

\section{Application to generalized Poisson--Boolean models}\label{sec_gen_boolean}
When the connection function $\varphi(t,t')$ is the characteristic function of some measurable set $\cS\subset (\bbR^d\otimes\bbM)^2$, then the RCM reduces the a generalized Boolean model and the connection random variables can be disregarded.  Given $\o \in \O_*$, the graph $G_d=G_d(\o)$ has vertex set $\hat\o$ and edges given by the pairs $\{x,y\}$ with $x\not =y$ in $\hat\o$ and $\big((x,m_x), (y,m_y)\big)\in\cS$. We discuss in the rest two important particular cases.

\subsection{Poisson--Boolean model}
\label{subsec:boolean}
A very well-studied model is the Poisson--Boolean model, where $\bbM=[0,\infty)$ and $\cS$ equals
\[\cS_{\rm B}:=\{ ((x,m_x), (y,m_y)) : |x-y|_2
\leq m_x+m_y\}\] (i.e.\ given $x\not=y$ in $\hat\o$ the pair $\{x,y\}$ is an edge of $G_d(\o)$ whenever the balls of radius $m_x$ and $m_y$ centered at $x$ and $y$, respectively, intersect). It is known (see~\cite{Hall85}) that if $\int \nu(\md m) m^d =+\infty$ then, for any $\rho>0$, $P_{\rho,\nu}$--a.s.\ the union of balls covers all $\bbR^d$. The natural hypothesis is then to assume, as we now do,  that 
\be\label{hall}0< \int \nu(\md m) m^d <+\infty\,. \en
Under this assumption, there exists $\l_c(\nu)\in (0,+\infty)$ with the following property: for $\rho<\l_c(\nu)$ the graph $G_d(\o)$   has   only bounded  connected components $P_{\rho,\nu}$--a.s., while for $\rho>\l_c(\nu)$ the graph $G_d(\o)$ has a unique  infinite connected component $\cC(\o) $  $P_{\rho,\nu}$--a.s.\ (see~\cite{Gouere08}, the existence of at most one infinite connected component follows also from \cite{Chebunin25}*{Theorem~12.1} since for any intensity, \eqref{eq:irreducible:basic} is satisfied).  By combining Theorem~\ref{teo_crossings} with the results of \cite{Dembin26} we get the following.
\begin{Theorem}\label{teo_booleano}
Assume \eqref{hall}
and consider the Poisson--Boolean model on the homogeneous PPP with intensity $\rho>\l_c(\nu)$. Then the following holds:
\begin{enumerate}
    \item\label{boolean:i} There exist $c_1,c_2>0$ such that, for all $\ell$ large enough, it holds
\[
P_{\rho,\nu} \left( \cN_{\ell}   \geq  c_1 {\ell}^{d-1}\right)\geq 1- \exp\left(- c_2 \,{\ell} ^{d-1}\right)\,,
\]
where $\cN_{\ell}$ denotes the maximal number of vertex-disjoint LR crossings of $\L_{\ell}$ included in the infinite cluster $\cC(\o)$ of $G_d(\o)$.
\item \label{boolean:ii} There exists $C>0$ such that property $P(C)$  holds $P_{\rho,\nu}$--a.s.\ and in particular $\kappa_{\rm eff} >0$.
\item \label{boolean:iii} Under the stronger assumption that $0<\int_0^{+\infty}\nu(\md m)m^{d+2}<+\infty$, the effective homogenized matrix $D(\rho)$ associated to the infinite cluster with unit conductances and the probability measure $P_{\rho,\nu}$ is well defined, it is diagonal and strictly positive definite.
\end{enumerate}
\end{Theorem}
\begin{proof}
\ref{boolean:i}  We  apply  \cite{Dembin26}*{Corollary~1.6}. To this aim 
  we take $n$ large to have $\nu([0,n])>0$. We  write $\nu|_{[0,n]}$ for the the restriction of $\nu$ to $[0,n]$ and we write $\nu_n$ for the probability measure $\frac{1}{\nu([0,n])}\nu|_{[0,n]} $.   Trivially the measure $\rho \cL \otimes \nu|_{[0,n]}$ equals the measure
  $ \nu([0,n]) \rho\cL\otimes \nu_n$.  
Hence, by \cite{Dembin26}*{Corollary~1.6},  we can fix $n$ large enough to have $\nu([0,n]) \rho>\l_c( \nu_n)$. Then,  Theorem~\ref{teo_crossings} gives a lower bound for the LR crossings of  the Boolean model given by the RCM associated to the triple $(\nu([0,n]) \rho, \nu_n, \1_{\cS_{\rm B}})$,  since the triple $(\l,\nu_n,\1_{\cS_{\rm B}})$ satisfies Assumption~\ref{assumere} for any 
$\l\in (\l_c(\nu_n),\rho)$. On the other hand,  since the measure $\rho \cL \otimes \nu $ dominates the measure $ \nu([0,n]) \rho\cL\otimes \nu_n$, it is simple to conclude by stochastic domination.

\ref{boolean:ii} The claim  follows from Theorem~\ref{teo3} and the domination pointed out for \ref{boolean:i} by monotonicity.

\ref{boolean:iii}  The diffusion matrix $D(\rho)$ is well defined since the model fits to the general abstract setting of e.g.~\cite{Faggionato25}. In this sense, we stress that the additional condition $\int_0^{+\infty}\nu(dm)m^{d+2}<+\infty$ implies   for both $k=0$ and $k=2$ that the expression 
$\sum _{x\in \hat \o: x\sim 0  } |x|^k  $ has finite expectation w.r.t.\ the Palm distribution $P_{\rho,\nu}(\cdot| 0 \in \cC(\o))$, which is one of the assumptions of \cite{Faggionato25} (the others can be easily checked). The proof that $D(\rho)$ is a strictly positive scalar matrix is obtained by combining Item~\ref{boolean:i} with Lemma~\ref{lemma_LB} and \cite{Faggionato25}*{Corollary~2.7} (similarly to the proof of Theorem~\ref{teo2}).
\end{proof}
Note that in the above Theorem~\ref{teo_booleano} we have not assumed anything similar to Assumptions~\ref{assumere}\ref{ass:finite:support}, in particular $\nu$ can have unbounded support. We stress that Theorem~\ref{teo_booleano} generalizes the result of \cite{Tanemura93} concerning the LR crossings for the case of a deterministic radius.

\subsection{Mott variable range hopping}\label{sec:Mott:vrh}

We conclude this section by discussing a relevant application to Mott v.r.h., the hopping mechanism at the basis of transport in amorphous solids as doped semiconductors.
In the generalized Boolean model corresponding to Mott v.r.h.\ with cutoff $\z>0$,   $\nu$ is a probability measure on $\bbR$  and the structure set $\cS=\cS_{\rm M}(\z)$ is given by 
\be\label{S_mott}
\cS_{\rm M}(\z):=\{ \left( (x,m_x), (y,m_y)\right)\in (\bbR^{d}\times\bbR) ^2 :  |x-y|+ w(m_x,m_y) \leq \z\}\,,
\en
where
\be\label{def_w} w(a,b)=|a|+|b|+|a-b|= \begin{cases} 2 \max\{|a|,|b|\} & \text{ if } ab\geq 0\,,\\
2|a-b| & \text{ if } ab <0\,.
\end{cases}
\en

Note that, since $w(a,b)\geq 2|a|$, any point with mark outside $(-\z/2,\z/2)$ is isolated. In particular, if $\nu\big((-\z/2,\z/2)\big)=0$,  then all points in the graph $G_d$ are isolated $P_{\rho,\nu}$--a.s. Excluding this trivial scenario, the model presents a non-trivial phase transition:
\begin{Lemma}\label{mott:phase:transition}
If $\nu \big((-\z/2,\z/2)\big)>0 $, then
   there exists $\l_c =\l_c(\nu,\z)\in (0,+\infty)$ with the following property: for $\rho<\l_c$ the graph $G_d$   has   only bounded  connected components $P_{\rho,\nu}$--a.s., while for $\rho>\l_c $ the graph $G_d$ has at least one unbounded connected component   $P_{\rho,\nu}$--a.s. 
\end{Lemma}
\begin{proof} 
 The random graph $G_d$ is included in the Poisson--Boolean model with deterministic radius $\z/2$ on the  PPP with intensity $\rho$, so for $\rho>0$ small enough, $P_{\rho,\nu}$-a.s.\ $G_d$ has no infinite cluster. On the other hand, fixing some $x\in (-\zeta/2,\zeta/2)\cap\supp\nu$, we have that $G_d$ contains the Poisson--Boolean model with deterministic radius $(\zeta/2-|x|)/2$ on the points with marks in $I_x=[x-(\zeta/2-|x|)/4,x+(\zeta/2-|x|)/4]$ (which form a PPP with intensity $\rho\nu(I_x)>0$). Indeed, for marks $a,b\in I_x$, we have
 \begin{align*}
 w(a,b)&{}\le 2\max\{|a|,|b|,|a-b|\}\le 2\max\{|x|+(\zeta/2-|x|)/4,(\zeta/2-|x|)/2\}\\
 &{}=\max\{\zeta/4+3|x|/2,\zeta/2-|x|\} \leq \z/2+|x|= \zeta-(\zeta/2-|x|).\end{align*}
 Hence, for $\rho$ large enough, $P_{\rho,\nu}$-a.s.\ $G_d$ has at least one  infinite cluster. The conclusion then follows by stochastic domination.
\end{proof}

It is immediate to check items \ref{ass:stationary}, \ref{ass:finite:support} and \ref{ass:symmetry} of Assumption~\ref{assumere}. By Lemma~\ref{mott:phase:transition} Assumption~\ref{assumere}\ref{ass:supercritical} is satisfied if $\nu\big((-\z/2,\z/2)\big)>0$ and $\l>\l_c$. To deal with Assumption~\ref{assumere}\ref{ass:irreducible} we  introduce the following.
\begin{Definition}\label{def_good} We say that $\nu$ is \emph{good} if $\nu\big((-\z/2,\z/2)\big)>0$ and, either $\supp\nu$ intersects at most one of $(-\z/2,0)$ and $(0,\z/2)$, or $A_+-A_-<\z/2$ where $A_-:=\sup\left( \supp\nu \cap (- \z/2,0)\right)$ and $ A_+:=\inf\left(\supp\nu \cap [0,\z/2)\right)$.
\end{Definition}

\begin{Lemma}\label{lem:mott:irr}If $\nu$ is good, then Assumption~\ref{assumere}\ref{ass:irreducible} is satisfied for any $\l>0$.
\end{Lemma}
\begin{proof}
As already observed points with marks outside $(-\z/2,\z/2)$ are isolated in $G_d$. Consider the  graph $\bar G_d$ given by  $G_d$ restricted to the points $(x,m_x)\in\o$ such that $|m_x|<\z/2$. Then, $\bar G_d$ is again a generalized Poisson--Boolean model with structure set $\cS_{\rm M}$ given by \eqref{S_mott},  now driven by a marked PPP with law $P_{\bar \rho,\bar\nu}$ where $\bar\rho=\rho \nu\left((-\z/2,\z/2)\right)>0$ and  $\bar \nu:=\nu\left(\cdot |(-\z/2,\z/2)\right)$. Since the unbounded connected components of $G_d$ and $\bar G_d$ coincide, our claim is proved, if we derive it for $\nu$ concentrated on  $(-\z/2,\z/2)$. 

By the above discussion, we now assume that $\nu$ is concentrated on  $(-\z/2,\z/2)$ and aim to verify Assumption~\ref{assumere}\ref{ass:irreducible} by using the criteria developed in \cite{Chebunin25} and recalled in Section~\ref{sec_model_results}. Take $m_x,m_y\in [0,\z/2)\cap \supp \nu$. Since $w(m_x,m_x)<\z$ and $w(m_x,m_y)<\z$, setting $t:=(x,m_x)$, $t':=(y,m_y)$ and letting $\xi$ be a  PPP with law $P_{\rho,\nu}$,  we have that $t$ and $t'$ are connected in $G(\{t,t'\}\cup\xi)$ with positive probability (it is enough that around the segment from $x$ to $y$ in $\bbR^d$ there are sufficiently dense points with marks around $m_x$). A similar argument holds with $m_x,m_y\in (-\z/2,0]\cap \supp \nu$. If $\supp\nu \cap (-\z/2,\z/2)$ contains points of different sign, $A_+-A_-<\zeta/2$ allows us to similarly prove that  $t$ and $t'$ are connected in $G(\{t,t'\}\cup\xi)$ with positive probability if $m_x$ and $m_y$ are sufficiently close to $A_-$ and $A_+$, respectively (thus ensuring that $w(m_x,m_y)< \z$). Putting these constructions together, 
we get \eqref{eq:irreducible}, so Assumption~\ref{ass:irreducible} is satisfied by \cite{Chebunin25}*{Proposition~5.1, Theorem~12.1}.   
\end{proof}

\begin{Remark} 
By the above, if $\nu$ is good,   then Assumption~\ref{assumere} is satisfied. In particular, if  $0\in\supp\nu$ (as relevant in physics), then $\nu$ is good and  Assumption~\ref{assumere} holds.
\end{Remark}

\begin{Remark}\label{misto}If $\supp \nu$ intersects both $(-\z/2,0)$ and $[0,\z/2)$ but $\nu$ is not good, then we can write the PPP with law  $P_{\rho,\nu}$ as superposition of independent PPPs with laws $P_{\rho_-,\nu_-}$ and $P_{\rho_+, \nu_+}$ where $\rho_-:=\rho\nu((-\infty,A_-])$, $\nu_-=\nu(\cdot|(-\infty,A_-]) $, $\rho_+:=\rho\nu([A_+,+\infty))$, $\nu_+=\nu(\cdot|[A_+,+\infty)) $. A.s.\ there is no edge in the RCM$(\rho,\nu,\varphi)$ between these two PPPs (and in particular $\l_c(\nu)=\min\{\l_c(\nu_-),\l_c(\nu_+)\}$) and for each one, Assumption~\ref{assumere}\ref{ass:irreducible} is verified by Lemma~\ref{lem:mott:irr}. In particular, one can apply our general approach to $P_{\rho_-,\nu_-}$ and $P_{\rho_+, \nu_+}$ separately.
\end{Remark}

\begin{Theorem}\label{teo_mott_vrh}
Assume that $\nu$ is good. Then, for any $\rho>\l_c$, the generalized Boolean model with structure set $\cS_{\rm M}$ built on the marked PPP with law $P_{\rho,\nu}$ satisfies \ref{boolean:i} and \ref{boolean:ii} of Theorem~\ref{teo_booleano} where $\cC(\o)$ is a.s.\ the unique infinite cluster. Moreover, the effective homogenized matrix $D(\rho)$ associated to $\cC(\o)$ with unit conductances and to $P_{\rho,\nu}$  is diagonal and strictly positive definite. 
\end{Theorem}
Theorem~\ref{teo_mott_vrh} is an immediate application of Theorems~\ref{teo_crossings}, \ref{teo2} and \ref{teo3} with $\varphi:=\1_{\cS_{\rm M}}$ and $\l\in (\l_c,\rho)$, hence we omit the proof.

\begin{Theorem}\label{teo_mott_vrh_2}
Suppose that $\nu\big( (-\z/2,\z/2)\big)>0$.
Then for any $\rho>\l_c(\nu)$ there exist $c_1,c_2>0$ such that, for all $\ell$ large enough, it holds
\be
P_{\rho,\nu} \left( \cN^*_{\ell}   \geq  c_1 {\ell}^{d-1}\right)\geq 1- \exp\left(- c_2 \,{\ell} ^{d-1}\right)\,,
\en
where $\cN^*_{\ell}$ denotes the maximal number of vertex-disjoint LR crossings of $\L_{\ell}$.
\end{Theorem}
\begin{proof} If $\nu$ is good, the claim  is an immediate consequence of Theorem~\ref{teo_mott_vrh} for what concerns the LR crossings of $\L_\ell$ included in the unique infinite cluster. If $\supp \nu$ intersects both $(-\z/2,0)$ and $(0,\z/2)$ and $A_+-A_-\ge\z/2$,  keeping the same as in Remark~\ref{misto}, both $\nu_-$ and $\nu_+$ are good and we can apply the previous result for good mark distributions to $P_{\rho_-,\nu_-}$ if $\l_c=\l_c(\nu_-)$ or to $P_{\rho_+,\nu_+}$ if $\l_c=\l_c(\nu_+)$.
\end{proof}

We point out that Theorem~\ref{teo_mott_vrh_2} is fundamental for the derivation of Mott's law in \cite{Faggionato26}  when the energy marks may have different signs (see Theorem~7.11 and Corollary~7.12 there). Additionally, Theorem~\ref{teo_mott_vrh_2} extends the result of \cite{Faggionato21}  (where marks have a fixed sign) for Mott~v.r.h.\ with a more direct and shorter proof.

\section{Preliminaries} \label{sec_preliminaries}
 \subsection{RCM driven by a PPP on \texorpdfstring{$\mathbb R^d\times \mathbb M$}{RdxM}} In Section~\ref{sec_model_results}, we dealt with the  graph $G_d$ in $\bbR^d$, since in applications this is the relevant object and marks just  model some local randomness in its construction. On the other hand, for our proofs,  it is more natural to deal with connections and graphs in $\bbR^d \times\bbM$.
 
  Given a configuration $\o\in \O$, 
 we can enumerate the points in $\o$ in  a measurable way \cite{Last18}*{Proposition~6.2}, i.e.\ the map $\O\to (\bbR^d\times \bbM)\cup\{\partial\}:\o \mapsto t_i(\o)$ is measurable, where $t_i(\o)$ is the $i$--th point of $\o$ if it exists, or the abstract point $\partial$ if  the $i$--th point  does not exist.
When we write $\o=\{t_1,t_2,\dots\}$ we mean that $t_1,t_2,\dots$ is the resulting enumeration. Recall that $J:=\{(i,j)\in \bbN_+\times\bbN_+\,:\, 1\leq i < j\}$.
 Given $\o=\{t_1,t_2,\dots\}\in\O$ and $\underline u \in [0,1]^J$, the graph $G(\o, \underline u)$ is defined as follows: it has 
  vertex set $\o$ and edge set given by the pairs $\{t_i,t_j\}$ of distinct vertices such that $u_{i,j}\leq \varphi (t_i, t_j)$ (we use the convention that $u_{i,j}:=u_{j,i}$ for  $1 \leq j <i$).

  Let $G=(V,E)$ be a graph as above  (i.e.\ $G=G(\o,\underline u)$), where $V$ and $E$ are respectively the vertex set and the edge set. If $\{t,t'\}$ is an edge, we write $t\sim t'$ and say that $t,t'$ are \emph{adjacent}. A path in $G$  is a string $(t_1,t_2,\dots, t_n)$ such that  $t_k\in V$ for all $k=1,2,\dots,n$ and  $\{t_k,t_{k+1}\}\in E$ for $k=1,2,\dots, n-1$.
 Given $A,B,D \subset  \bbR^d\times\bbM$ we say that there is a path from $A$ to $B$ in $D$ if there is a path  $(t_1,t_2,\dots, t_n)$ in $G$ such that $t_1\in A$, $t_n\in B$ and $t_k\in D$ for all $k=1,2,\dots, n$. In this case we write ``$A\leftrightarrow B$ in $D$'' or also ``$A\leftrightarrow B$ in $D$ in $G$'' when we need  to stress that we deal with the graph $G$. If $D=\bbR^d\times \bbM$, it is omitted from the notation.

The main object of our investigation is the random graph $G(\o, \underline u)$ defined on the probability space $\O\times [0,1]^J$ with the probability measure $P_{\rho, \nu} \otimes Q$.
    This is the RCM  on $\bbR^d\times \bbM$ driven by the PPP with law $P_{\rho, \nu}$.
 Instead of $P_{\rho, \nu}$ we will consider also other PPPs on $\bbR^d\times \bbM$, hence not necessarily with intensity measure $\rho\cL\otimes \nu$.

 While $G(\o, \underline u)$ is a fixed graph for $(\o,\underline u)\in \O\times [0,1]^J$, we will use $G(\z)$ to denote a random graph as follows.
 Given a locally finite set of points $\z\subset \bbR^{d}\times\bbM$ (possibly random) we write $\cG(\z)$ for the RCM on $\z$, i.e.~the random graph  with vertex set $\z$ and edge set obtained by putting an edge  between two vertices $t\not= t'$ in $\z$ with probability $\varphi (t,t')$ independently  from the other pairs of vertices in $\z$.

 \subsection{FKG inequality}\label{sec_FKG}
 To benefit of the FKG inequality we need to switch to a different space.  As in \cite{Last17}, we define $(\bbR^d\times\bbM)^{[2]}$ as the set $\{ A\subset \bbR^d\times \bbM\,:\, \sharp A=2\}$.
When equipped with the Hausdorff metric this set is a Borel subset of a complete separable metric space \cite{Last17}. 
We set $\bbW:=(\bbR^d\times\bbM)^{[2]}\times[0,1]$ (trivially, an element of $\bbW$ is a pair $\big( \{t,t'\}, u\big)$ with $t\not =t'$ in $\bbR^d\times \bbM$ and $u\in [0,1]$). 
We then define $\cN(\bbW)$ as the family  of locally finite subsets of $\bbW$. Then $\cN(\bbW)$ is a measurable space with $\s$--algebra of measurable sets generated by the sets $\{ \z\in \cN(\bbW)\,:\, \sharp (\z\cap B) =n\}$ as $B$ varies among the Borel sets of $\bbW$ and $n$ varies in $\bbN$.

On $\bbW$, we define the partial order $\leq$  as follows: $\z\leq\z'$ if  for any $(\{t,t'\},u)\in\z$ it holds
$(\{t,t'\},u')\in\z'$ for some $u'\in[0,u]$.
A measurable function  $f:\cN(\bbW)\to\bbR$ is  called \emph{increasing}  if  $f(\z)\leq f(\z')$ whenever   $\z\leq \z'$. An event  $A$  in $\cN(\bbW)$ is called increasing if $\1_A$ is increasing.

Consider the map
\[ \Psi:\O \times [0,1]^J\to\cN(\bbW):  (\o,\underline{u} ) \mapsto 
\Big\{ \big(\{t_i, t_j\}, u_{i,j}\big): (i,j)\in J,\;j \leq \sharp\o \Big\}\,,
\]
where $\o=\{t_1,t_2,\dots\}$. 
\begin{Lemma}(FKG inequality) \label{lemma_FKG} 
The image $P$ of $P_{\rho,\nu}\otimes Q$ by $\Psi$ satisfies the FKG inequality w.r.t.\ the partial order $\leq $  on $\cN(\bbW)$: writing $E[\cdot]$ for the expectation w.r.t.~$P$, it holds
\begin{equation} \label{eq:FKG}
 E[fg]\geq E[f]E[g],\end{equation}
for all measurable bounded increasing functions  $f,g:\cN(\bbW)\to\bbR$.
\end{Lemma}
\begin{proof}
Given $\z\in\cN(\bbW)$ we set $\eta=\eta(\z):=\{ t\in \bbR^d\times\bbM\,:\, (\{t,t'\},u)\in \z \text{ for some } t',u\}$. Then, as explained below, we have
\[
E[fg]=E\big[E[fg|\eta]\big]\geq E\big[E[f|\eta]E[g|\eta]\big]\geq E\big[E[f|\eta]\big] E\big[E[g|\eta]\big]=E[f]E[g]\,.
\]
Indeed, the first inequality follows from the FKG inequality for independent uniform random variables with the partial order naturally induced by the standard order on real numbers ($f$ and $g$ are both decreasing functions of the connection variables $u$).  The second inequality follows from the FKG inequality for PPPs (cf.~\cite{Last18}*{Theorem~20.4}) as  $E[f|\eta]$ and $E[g|\eta]$ are both increasing functions of $\eta$ w.r.t.\ inclusion (see below), while under $P$ $\eta$ is a PPP with law $P_{\rho,\nu}$. To check the last monotonicity, we observe that, given $\eta,\eta'\in \O$ with $\eta\subset\eta'$, we can realize a coupling between the connection variables associated to $\eta$ and the ones associated to $\eta'$ so that for all $x\not =y$ in $\eta\subset\eta'$ the connection variable for the pair $\{x,y\}$ is the same for $\eta$ and for $\eta'$. Then, the corresponding configurations in $\cN(\bbW)$ are in increasing order w.r.t.~$\leq $ (where each configuration is given by the collections of triples $(\{x,y\}, u)$, $u$ being the connection variable associated to $\{x,y\}$). The same order is then maintained when applying $f$ or $g$.
\end{proof}

\subsection{Irreducibility}\label{subsec:free:ass} 
In this section we show that to prove Theorem \ref{teo_crossings} we can assume w.l.o.g.\ that
\begin{equation}
    \label{eq:supercritical}
    \nu\Big(\Big\{m\in\bbM:\bbP\big((0,m)\leftrightarrow \infty\text{ in } G\big(\bbX \cup  \{(0,m)\}\big)\big)>0\Big\}\Big)=1
    \end{equation}
    where $\bbX$ is a marked PPP on $\bbR^d\times\bbM$ with law $P_{\l,\nu}$. We point out that \eqref{eq:supercritical} does not follow from Assumptions~\ref{assumere}. Indeed, it is easy to exhibit a counterexample e.g.\ by means of Mott v.r.h.\ with cutoff introduced in Section~\ref{sec:Mott:vrh}.

\begin{Lemma} 
\label{lem:supercritical:assumption}
Let
\begin{equation}
\label{eq:def:A}
A:= \Big\{m\in\bbM:\bbP\Big((0,m)\leftrightarrow \infty\text{ in } G\big(\bbX \cup  \{(0,m)\}\big)\Big)=0\Big\}\,.
\end{equation}
For $P_{\l,\nu}\otimes\bbQ$--a.a.~$(\o,\underline u)$ the following holds: for each $(x,m)\in \o$ with $m\in A$ the cluster of $(x,m)$ in $G(\o,\underline u)$ is finite.
\end{Lemma}
\begin{proof}
   We recall that the Palm distribution associated to the marked PPP with law $P_{\l,\nu}$ is the law of the  marked point process $\bbX\cup\{(0,M)\}$ where  $\bbX$ is a marked PPP with law $P_{\l,\nu}$  and  $M$ is an independent  random variable  with law $\nu$. 
  We define  $\cA$ as   the set of $(\o,\underline u)\in \O_*\times[0,1]^J$ such that $(0,m)\in \hat \o$, $m\in A$ and the cluster of $(0,m)$ in $G(\o,\underline u)$ is infinite. Fix $L>0$ and let $\Lambda=[-L,L]^d$. By Campbell's formula (cf.~\cite{Daley88}*{Eq.~(12.2.4)}  or \cite{Franken82}*{Theorem~1.2.8} for marked point processes) applied to the function $f(x, (\o, \underline u))= \1_{\L} (x) \1_{\cA
  }( \o, \underline u) $ and by the above characterization of the Palm distribution we have\footnote{As in  the displayed formula, for a  generic event $E$, we will sometimes   write $\1(E)$ instead of $\1_E$ in order to improve readability.} 
  \[ \int _{A} d\nu(m) \bbP\Big(  (0,m)\leftrightarrow \infty \text{ in } G \big(\bbX\cup\{(0,m)\}\big)  \Big)=\frac{1}{\l (2L)^d}
  \bbE\Big[ \sum_{ \substack{(x,m_x)\in \bbX :\\ 
  x\in \L, m_x\in A}} \1\big(x \leftrightarrow 
 \infty \text{ in } G(\bbX) \big)\Big]\,.
  \]
  Since, by definition of $A$, the l.h.s.\ above is zero, from the above identity we get that $\bbP$--a.s.\ $x \not\leftrightarrow 
 \infty \text{ in } G(\bbX)$ for any  $(x,m_x)\in \bbX $ with $x\in \L$ and $m_x\in A$. By  varying $L>0$, we get our claim.
\end{proof}

By Lemma~\ref{lem:supercritical:assumption}, given a marked PPP $\bbX$ with law $P_{\l,\nu}$,  if we remove  from $G(\bbX)$  all vertices  $(x,m_x)$ with $m_x\in A$ and the associated edges, the resulting graph has the same infinite clusters as $G(\bbX)$. By Assumption~\ref{assumere}\ref{ass:supercritical}, we conclude that $\nu(A)<1$ with $A$ from \eqref{eq:def:A}. Set $\lambda'=\lambda\nu(A^c)$ and $\nu':=\nu(\cdot|A^c)$. Let us verify that the triple $(\l',\nu',\varphi)$ satisfies  Assumption~\ref{assumere}~\ref{ass:irreducible} and \ref{ass:supercritical}.  The graph obtained above by removing vertices  $(x,m_x)$ with $m_x\in A$ and projected on $\bbR^d$ is just the RCM$(\l',\nu', \varphi)$ and, by Assumption~\ref{assumere}\ref{ass:supercritical}, it has an unbounded connected component $P_{\l',\nu'}\otimes Q$--a.s. Hence, the triple $(\l',\nu',\varphi)$ satisfies Assumption~\ref{assumere}~\ref{ass:supercritical}. 
 
As discussed in Section~\ref{sec_model_results}, in order to verify Assumption~\ref{assumere}\ref{ass:irreducible} for $\lambda',\nu'$, by \cite{Chebunin25}, it suffices to verify \eqref{eq:irreducible} for the new parameters $\l',\nu'$.
  Fix $(t,t')\in(\bbR^d\times A^c)^2$ with $t\neq t'$. Then, by \eqref{eq:def:A} and translation invariance, $\bbP(t\leftrightarrow\infty\text{ in }G(\bbX\cup\{t\}))>0$ and similarly for $t'$. By the FKG inequality, this gives that $\bbP(t\leftrightarrow\infty,t'\leftrightarrow\infty\text{ in }G(\bbX\cup\{t, t'\}))>0$ and, by Assumption~\ref{assumere}\ref{ass:irreducible}, we obtain 
 \be\label{aereo}
 \bbP(t\leftrightarrow t'\leftrightarrow \infty \text{ in }G(\bbX\cup\{t,t'\}))>0\qquad \forall (t,t')\in(\bbR^d\times A^c)^2\,.
 \en
Given a locally finite subset $\eta$ in  $\bbR^d\times\bbM$ and given $t\not=t'$ in  $\eta$, we define 
\[f(t,t',\eta):=\1(t,t'\in\bbR^d\times A^c)\bbP 
\left(
\bbR^d\times A  \leftrightarrow t\leftrightarrow t'\leftrightarrow \infty \text{ in }G(\eta)\right)\,.\]
Above the probability is w.r.t.\ the connection variables. Thanks to Lemma~\ref{lem:supercritical:assumption}, we have
\[\bbE\Big[\sum_{t\in \bbX}\sum_{t'\in\bbX:t\not=t'} f(t,t',\bbX)\Big]=0\]
(we stress that now the expectation is w.r.t.\ both the PPP $\bbX$ and the connection variables). By the multivariate Mecke equation (cf.~\cite{Last18}*{Theorem~4.4}), we conclude that  for $(\l\cL\otimes \nu)^{\otimes 2}$ a.a.~$(t,t')\in (\bbR^d\times A^c)^2$ it holds $\bbE[f(t,t',\bbX\cup\{t,t'\})]=0$, which can be rewritten as $\bbP\left(
\bbR^d\times A  \leftrightarrow t\leftrightarrow t'\leftrightarrow \infty \text{ in }G(\bbX\cup\{t,t'\})\right)=0$. Combining this result with \eqref{aereo}, we get $\bbP(t\leftrightarrow t'\text{ in }G((\bbX\setminus(\bbR^d\times A))\cup\{t,t'\}))>0$  for $(\l\cL\otimes \nu)^{\otimes 2}$--a.a.~$(t,t')\in (\bbR^d\times A^c)^2$. The last property is equivalent to \eqref{eq:irreducible} for $\l',\nu'$.

We have proved that  the triple $(\l',\nu',\varphi)$ satisfies Assumption~\ref{assumere}. Clearly,   \eqref{eq:supercritical} is valid  with $\nu$ and $\bbX$ replaced by $\nu'$ and a marked PPP with law $P_{\l',\nu'}$, respectively. Since 
RCM$(\l',\nu',\varphi)$ can be thought of as a subgraph of RCM$(\l,\nu,\varphi)$, this completes the reduction of Theorem~\ref{teo_crossings}, showing that we can assume \eqref{eq:supercritical} w.l.o.g.

\subsection{Additional notation}
\label{subsec:notation}
We conclude with some additional notation and conventions used throughout this work. Since in the next sections $\rho$ will denote a generic density, we will prove Theorem~\ref{teo_crossings} for $\l'>\l$ instead of $\rho>\l$.
\begin{itemize}
\item We fix a triple $(\lambda,\varphi,\nu)$ satisfying Assumptions~\ref{assumere} and \eqref{eq:supercritical}. We also fix 
$\lambda'>\lambda$.
\item W.l.o.g.\ we assume ${\ell}_*:=1$ in Assumption~\ref{assumere}\ref{ass:finite:support}.
\item We set $\L_r=[-r,r]^d$ and  $\L_r(x):=x+ \L_r$, for $r\ge 0$ and $x\in\bbR^d$. 
\item We set $\partial_*\L_r:=\{ x\in \bbR^d\,:\, |x|_\infty \in [r-1,r]\}$ and  $\partial_* \L_r(x):=x+\partial_*\L_r$, for $r\ge 1$.
\item For $r\ge 1$, $\sigma\in\{-1,1\}^d$ and $i\in\{1,\dots,d\}$, we define the \emph{quarter-face}
\begin{equation}\label{eq:def:F}
F_r^{i,\sigma}=\left\{(x_1,\dots,x_d)\in\bbR^d:\sigma_i x_i\in[r-1,r]\text{ and }\sigma_j x_j\in[0,r]\text{ for }j\in\{1,\dots,d\}\setminus \{i\}\right\},
\end{equation}
omitting $i,\sigma$, if $i=1$ and $\sigma=(1,\dots,1)$.
\end{itemize}

\section{Connections, seeds and sprinkling}\label{sec_connections}

In view of Definition~\ref{def_G}, we can assume (here and in what follows) that $\varphi(t,t)=0$ for any $t\in \bbR^d\times\bbM$. Set 
$\bar\varphi (t,t'):=1-\varphi(t,t')$ for all $t,t'\in \bbR^d\times\bbM$. Given a finite set $C\subset \bbR^{d}\times\bbM$ and given $t\in \bbR^{d}\times\bbM$ we define
 $\bar \varphi (t,C):=\prod _{t'\in C}\bar \varphi (t,t')$. Note that $\bar \varphi(t,C)=0$ for all $t\in C$.  The probabilistic interpretation of $\bar\varphi(t,C)$ is the following. Consider independent Bernoulli random variables parametrized by the pairs $(t,c)$ with $c\in C$, where the $(t,c)$-random variable equals $1$ with probability $\varphi(t,c) $. Then $\bar\varphi(t,C)$ is the probability that all these random variables are zero.
  
The next lemma allows creating connections with the help of sprinkling (changing the intensity of the PPP). The core idea behind it goes back to the classical work \cite{Aizenman83}.

\begin{Lemma}[Sprinkling] \label{lemma_udine} 
Fix a finite set $A\subset \bbR^d\times\bbM$ and 
Borel sets $B\subset R\subset\bbR^{d}\times\bbM$ with $B$ bounded. Let $\rho, \d>0$ and let $\bbX,\tilde\bbX,\bbY$ be PPPs on $R$ with intensity measures $\rho \mu(\md t)$, $\d\mu(\md t)$ and  $\rho \bar\varphi(t,A) \mu(\md t)$, respectively, where $\mu$ is the restriction of $\cL\otimes \nu$ to $R$.  We suppose that $\tilde \bbX$ is independent from $\bbY$.  
Then, for any $u\geq 0$,
\begin{equation}
\label{fiorellino} \bbP\left(A\leftrightarrow B\text{ in }G\big(A,\tilde\bbX\big)\cup G\big(\tilde\bbX\cup\bbY\big)\right) \geq 
(1- e^{-\d u}) \Big [ 1- e^{\rho u} \bbP\Big( A\not \leftrightarrow B  \text{ in } G\big(\bbX\cup A\big) \Big) \Big]\,,
\end{equation}
where $G(A,\tilde\bbX)$ denotes 
the bipartite RCM with connection function $\varphi$, i.e.\ the graph with vertex set $A\cup \tilde \bbX$ where an edge $\{t,t'\}$ is created with probability $\varphi(t,t')$ independently when   varying $t\in A$ and $t'\in\tilde\bbX$.
\end{Lemma}
   
\begin{proof} In what follows, to shorten, we will write PPP$[\mathfrak{m}]$ for a PPP  with intensity measure $\mathfrak{m}$. If $A\cap B\not= \varnothing$, then \eqref{fiorellino} is trivially true. Let us then restrict to $A\cap B=\varnothing$. Given a realization of $\bbX$, a point $t\in \bbX$ is adjacent to $A$ in $G(\bbX\cup A)$  with probability  $1-\bar \varphi(t,A)$ (the probability is w.r.t.\ the connection random variables) and this takes place independently for each point $t\in \bbX$. Hence, if we set
$\xi:=\{t\in \bbX \,:\, t\not\sim A \text{ in } G(\bbX\cup A) \}$ and $\eta:=\{t\in \bbX \,:\, t\sim A\text{ in } G(\bbX\cup A)\}$, we get that $\xi$ and $\eta$ are independent, $\xi\sim\text{PPP}[\rho \bar \varphi (t, A) \mu(\md t) ]$ and $\eta\sim\text{PPP}[n(\md t)]$, where  $n(\md t):=\rho ( 1- \bar \varphi (t,A) ) \mu(\md t)$. Since, in \eqref{fiorellino}, we may assume that $\bbX$ and $\tilde\bbX$ are independent, we may set $\bbY=\xi$ and $G(\bbY)$ to be the restriction of $G(\bbX)$ to the vertex set $\xi$.

We define
\[ \cW:=\{ t\in \bbY\,:\, t\leftrightarrow B \text{ in } G(\bbY)\}.\] 
We observe that (see the  comments below)
\be\label{praga}
\begin{split}
&\bbP\left(A\not\leftrightarrow B\text{ in }G(\bbX\cup A)\right)=\bbP\left(B\cap\eta=\varnothing,\not\exists t\in \eta:t\sim \cW\text{ in }G(\bbX)\right)\\
&=\bbE\left[\bbP(B\cap \eta=\varnothing,\not\exists t\in\eta:t\sim\cW\text{ in }G(\{t\}\cup\cW))|\cW)\right]=
\bbE\left[e^{-\rho Z(\cW)}\right],
\end{split}
\en
where, for $W\subset R$ finite, we set
\[Z(W)=
 \int_R\Big( 1-\bar\varphi(t,W)
\1_{R\setminus B }(t) \Big)  ( 1- \bar \varphi (t,A) ) \mu(\md t).\]
We explain the last identity in  \eqref{praga}. The set $\cW$ is  independent from $\eta$. Hence,  we can take as  version of the conditional probability in the second line of \eqref{praga}  the function given by $p:=\bbP(B\cap \eta=\varnothing,\not{\exists} t\in\eta:t\sim W\text{ in }G(\{t\}\cup W))$ on the event $\{\cW=W\}$. To compute $p$ we   write 
$p=\bbP\big( (\eta\cap B)\cup \eta' = \varnothing\big)$ where $\eta':=\{t\in \eta\setminus B: t\sim W \text{ in }G(\{t\}\cup W 
)\}$.
On the other hand, $\eta\cap B$ and $\eta'$ are independent,  $\eta\cap B\sim\text{PPP}[\1_B (t) n(\md t)]$ and $\eta'\sim\text{PPP}[(1-\bar\varphi(t,W)) 
\1_{R\setminus B }(t) n(\md t)]$. Hence, their union $(\eta\cap B)\cup \eta'$ is a $\text{PPP}[\rho \hat n(\md t)]$, where
  $ \hat n(\md t):=  \big( 1-\bar\varphi(t,W)
\1_{R\setminus B }(t) \big)  ( 1- \bar \varphi (t,A) ) \mu(\md t)$. This implies that 
$p=\bbP\big( (\eta\cap B)\cup \eta'= \varnothing\big)= e^{- \rho\int_R \hat n (\md t)}= e^{-\rho Z(W)}$, thus leading to \eqref{praga}.

We now claim that 
\be\label{budapest}
\bbP\left(A\leftrightarrow \cW\cup B\text{ in }G\big(A,\tilde\bbX\big)\cup G\big(\tilde\bbX\cup\cW\big)\right)\geq 1-\bbE\left[e^{-\delta Z(\cW)}\right]\,,
\en
where the union of two graphs is the graph with vertex and edge sets given by the unions of the respective vertex and edge sets. By conditioning as in \eqref{praga} and using that a.s.\ $\bbY$ (and therefore $\cW$) is disjoint from $A$, \eqref{budapest} follows from the observation that, given a finite $W\subset R$ disjoint from $A$, it holds that
\begin{multline}\label{ritardo}\bbP\left(A\leftrightarrow W\cup B\text{ in }G\big(A, \tilde\bbX\big)\cup G\big(\tilde\bbX\cup W\big)\right)\\
\ge \bbP\left(\exists t\in\tilde \bbX:t\sim A\text{ in }G(\{t\}\cup A)\text{ and }\left(t\in B\text{ or }t\sim W\text{ in }G(\{t\}\cup W)\right)\right)=1-\bbE\left[e^{-\delta Z(W)}\right].\end{multline}
Let us justify  the last identity. Since $\tilde \bbX\sim\text{PPP}[\d \mu(\md t)]$,
the set of $t\in \tilde \bbX$ satisfying the request in the second line of \eqref{ritardo} is given by the union  of the sets 
\[ \Big\{t\in \tilde\bbX\cap B: t\sim A\text{ in }G(\{t\}\cup A)\Big\},
\quad 
\Big\{t\in \tilde \bbX\setminus  B: t\sim A\text{ in }G(\{t\}\cup A) \text{ and } t\sim W\text{ in }G(\{t\}\cup W)\Big \}\,,
\]
which are  independent and given by  $\text{PPP}[\1_B (t) (1-\bar \varphi (t,A)) \d \mu(\md t)]$ and $\text{PPP}[\1_{R\setminus B }(t) (1-\bar \varphi (t,A)) (1-\bar\varphi(t,W) )\d \mu(\md t)]$, respectively.
From this we conclude that the union is a $\text{PPP}[\d \hat n (\md t)]$.

Finally, applying Markov's inequality twice, for any $u\ge 0$, we get
\begin{equation}\label{airfrance}
\begin{split}
1-\bbE\left[e^{-\delta Z(\cW)}\right]&{}\ge \left(1-e^{-\delta u}\right)\bbP(Z(\cW)> u)\\
&{}=\left(1-e^{-\delta u}\right)\left(1-\bbP(Z(\cW)\le u)\right)\ge \left(1-e^{-\delta u}\right)\left(1-e^{\rho u}\bbE\left[e^{-\rho Z(\cW)}\right]\right).
\end{split}
\end{equation}
Combining these results concludes the proof of the lemma, since
\be\label{musica}\bbP\left(A\leftrightarrow B\text{ in }G\left(A,\tilde\bbX\right)\cup G\left(\tilde\bbX\cup\bbY\right)\right)
\geq
\bbP\left(A\leftrightarrow \cW\cup B\text{ in }G\left(A,\tilde\bbX\right)\cup G\left(\tilde\bbX\cup\cW\right)\right).
\en
Indeed, it is enough to  combine  \eqref{musica}, \eqref{budapest}, \eqref{airfrance} and  finally \eqref{praga} in that order.
\end{proof}

\begin{Definition}[Local uniqueness]
\label{def:local:uniqueness} Fix $r<s-1$ and $x\in\bbR^d$. Let $G=(\omega,E)$ be a random graph with $\omega\in\Omega$. Let $G'$ be the restriction of $G$ to the vertex set $\omega\cap (\Lambda_s(x)\times\bbM)$. We say that $\cA(r,s,x)$ occurs for $G$, if,  
for any two paths $\g,\g'$  in $G'$
from $\L_r(x) \times\bbM$ to $\partial_*\L_s(x)\times\bbM$, there is a path in  $G'$ from a vertex of $\g$ to a vertex of $\g'$. We set $\cA(r,s):=\cA(r,s,0)$.
\end{Definition}
We emphasise that we do not require and will not prove \emph{strong local uniqueness}, namely e.g.\ $\liminf_{N\to\infty}P_{\rho,\nu}\otimes Q( \cA(3N,4N))=1$. If this result were available a priori, it would greatly facilitate the proof, in particular essentially removing the need of Section~\ref{sec_tanemura}, as discussed in Section~\ref{subsec:outline}. However, proving such strong local uniqueness often requires a significant amount of work even after Theorem~\ref{th:slab} is proved. This is why we have opted for the simpler approach presented here.

Before moving on, we need to fix a few scales for the rest of the section. Recall that $\lambda,\lambda',\varphi,\nu,d$ are fixed, so all quantities may depend on them. Further choose 
\begin{equation}
\label{eq:scales}
N\gg n\gg K\gg k\gg1/\varepsilon\gg 1/c\gg 1,
\end{equation}
that is, $c$ is small enough; $\e$ is small enough depending on $c$; $k$ is large enough depending on $c,\varepsilon$; $K$ is large enough depending on $c,\varepsilon,k$ and so on. The precise way to choose these scales is detailed in the proofs of Lemmas~\ref{lemma_scale},~\ref{lemma_crescita} and~\ref{lem:seed:to:seed} below (in particular, all our statements should be understood to hold under a suitable choice of the above scales in \eqref{eq:scales}, fixed once and for all in the manner  specified  throughout  the text). 
 We now isolate a subset of good marks $\bbM_g\subset\bbM$ depending on the scale $\e$. Given  a PPP  $\bbX$ on $\bbR^d\times\bbM$ with intensity measure $\l \cL \otimes \nu$  (i.e.~with law $P_{\l,\nu}$), we set
\be\label{set_good_marks}
\bbM_g:=\Big\{\,m\in\bbM\,:\, 
\bbP
   \big (\, (0,m)\leftrightarrow \infty \text{ in } G\big(\bbX\cup  \{(0,m)\}\big)\, \big) \ge \varepsilon \,\Big\}\,.
\en
\begin{Lemma}\label{lemma_scale}  Let $\bbX$ be a  PPP   on $\bbR^d\times\bbM$ with intensity measure $\l \cL \otimes \nu$.
Then, we can assume that $\nu(\bbM_g)>1-c$ and, for the RCM $G(\bbX)$,
\begin{align}
    \label{eq:uniqueness:k}
    &\bbP \big(\cA(k,K)\big)\ge 1-e^{-1/\varepsilon},\\
    \label{eq:uniqueness:n}
    &\bbP \big(\cA(n,N)\big)\ge 1-e^{-1/\varepsilon},\\
  \label{eq:1arm:coarse} 
  &\bbP \big(\Lambda_k\times\bbM \leftrightarrow\partial_*\Lambda_N \times \bbM \text{ in } G(\bbX)\big)\ge 1-e^{-1/\varepsilon}.
\end{align}

\end{Lemma}
\begin{proof}
    By \eqref{eq:supercritical}, given $c\in(0,1)$, we can fix $\e>0$ small enough so that $\nu(\bbM_g)>1-c$.

By Assumption~\ref{assumere}\ref{ass:supercritical}, there exists $r_0$ depending on $\e$ such that $\bbP(\L_{r_0}\times\bbM \leftrightarrow \infty \text{ in } G(\bbX) )\geq 1-e^{-1/\e}$. Hence, \eqref{eq:1arm:coarse} is satisfied whenever $r_0\leq k < N-1$.

As in \cite{Dembin26}*{Lemma~6.3}, for $r<s-1$ we define $\cE_{r,s}$ as the event that the graph $G(\bbX)$ restricted to $\L_s\times\bbM$ has at least two connected components intersecting both $\L_r\times\bbM$ and $\partial_*\L_s\times\bbM$. Note that $\cA(r,s)\supset\cE_{r,s}^c$. Moreover, if the event  $\bigcap _{s=r+1}^\infty \cE_{r,s}$ occurs, then the finite family of the clusters of vertices in $\L_r\times\bbM$ must contain at least two disjoint infinite clusters. By Assumption~\ref{assumere}\ref{ass:irreducible} we then get that 
$\lim_{r\to+\infty} \lim_{s\to+\infty}\bbP( \cE_{r,s})=\bbP\left(\bigcup_{r=1}^\infty \big(\bigcap _{s=r+1}^\infty \cE_{r,s}\big)\right)=0$
and therefore $\lim_{r\to+\infty} \lim_{s\to+\infty}
\bbP(\cA(r,s))= 1$. This easily allows to take $r_0\leq k \ll K\ll n\ll N$ satisfying \eqref{eq:uniqueness:k}, \eqref{eq:uniqueness:n} and \eqref{eq:1arm:coarse}.
    \end{proof}

\begin{Definition}[Seed]
\label{def:seed}
A set $S\subset\bbR^{d}\times\bbM$ is a \emph{seed}, if the following conditions hold:
\begin{itemize}
    \item $S\subset \Lambda_{n-K}(z)\times \bbM_g$ for some $z\in\bbR^d$,
    \item $\sharp S=l:=\lceil 1/\varepsilon^2\rceil$,
    \item for all distinct $s=(x,m)\in S,s'=(x',m')\in S$, $\Lambda_K(x) \cap \Lambda_K(x')=\varnothing$.
\end{itemize}    
\end{Definition}
Let us remark that the notion of seed in Definition~\ref{def:seed} is not to be confused with the `seed' of \cite{Grimmett90}. Rather our seed should be thought of as analogous to the set $S$ in the renomalization scheme of
\cite{Duminil-Copin21}. Even though the latter is called `seedless' (because it does not use the seeds of \cite{Grimmett90}), the sets $S$ of \cite{Duminil-Copin21} de facto replace `seeds' in \cite{Grimmett90} in the sense that they are the starting point for further growth. In view of this, we still call them seeds, but in a broader sense.

The following lemma is an adaptation to our setting of \cite{Duminil-Copin21}*{Lemma~2}:
\begin{Lemma}[Seed to quarter-face]
\label{lemma_crescita} 
Let $S\subset\Lambda_{n-K}\times\bbM$ be a seed.  Let $\bbX$ be a  PPP   on $\bbR^d\times\bbM$ with intensity measure $\l \cL \otimes \nu$. Then, recalling \eqref{eq:def:F}, we have
\begin{equation}
    \label{eq:S:to:FN}\bbP  \Big(S\leftrightarrow  F_N\times\bbM \text{ in  }\Lambda_N\times \bbM 
    \text{ in } G( S\cup \bbX))\Big )\ge 1-e^{-c/\varepsilon}.
\end{equation}
\end{Lemma}
\begin{proof}
Let $S=\{s_i:i\in\{1,\dots,l\}\}$ where $s_i=(x_i,m_i)$.
We realize the random graph $G(\bbX\cup S)$ so that $G(\bbX)$ is the restriction of $G(\bbX\cup S)$ to the vertex set $\bbX$. For $i\in\{1,\dots,l\}$, we set $Q_i:=\Lambda_K(x_i)$ and $Q_i':=\L_k(x_i)$ and introduce the events
\begin{align*}
\cE_i&{}:= \left\{ (x_i,m_i) \leftrightarrow   \partial_*Q_i\times \bbM \text{ in  } G(\bbX\cup S)  \right\}\cap \cA(k,K,x_i)\,,\\
\cB_i&{}:=\left\{ Q_i'  \times \bbM  \not \leftrightarrow  \partial_*\L_N\times\bbM \text{ in  } G(\bbX)  \right\}\,,
\end{align*}
where the event $\cA(k,K,x_i)$ refers to the graph $G(\bbX)$.

By \eqref{eq:uniqueness:k} and since $m_i\in \bbM_g$ by Definition~\ref{def:seed}, we have $\bbP(\cE_i)\ge \varepsilon-e^{-1/\varepsilon}\ge \varepsilon/2$ for $\e$ small enough and any $i\in\{1,\dots,l\}$. Since the boxes $Q_i$ are disjoint and  each event $\cE_i$ depends only on vertices and edges in $Q_i\times\bbM$, the events $\cE_i$ are independent. Hence we can bound
\begin{equation}
\label{eq:cup:Ei}
\bbP\left(\bigcup_{i=1}^l \cE_i\right)\geq 1-\left(1-\frac{\e}{2}\right)^l \geq 1-  e^{-1/(2\e)}\ge 1-e^{-2c/\varepsilon},
\end{equation}
since $0\leq 1-x\leq e^{-x}$ for $0\leq x\leq 1$ and choosing $c\le 1/4$.

To estimate $\bbP(\cB_i)$ one cannot apply \eqref{eq:1arm:coarse} directly, because $Q_i'$ is not necessarily centered at $0$. As in \cite{Duminil-Copin21}, we fix $j(i)\in\{1,\dots,d\}$ and $\sigma(i)\in\{\pm1\}^d$, so that $F^i=x_i+F^{j(i),\sigma(i)}_N$ (recall \eqref{eq:def:F}) does not intersect the interior of $\L_{N-1}$. 
Note that the event $\cD_i:=\{Q_i' \times\bbM \leftrightarrow F^i\times\bbM \text{ in } \L_N(x_i)\times\bbM \text{ in }G(\bbX)\}$ is disjoint from $\cB_i$. Hence,
\begin{align}
\nonumber \bbP (\cB_i)&{}\leq \bbP\left( \cD_i^c\right)=\bbP\left(\Lambda_k\times\bbM\not\leftrightarrow F^{j(i),\sigma(i)}_N\text{ in }\Lambda_N\times \bbM\text{ in }G(\bbX)\right)\\
\label{eq:sqrt:trick}&{}=\left(\prod_{j\in\{1,\dots,d\}}\prod_{\sigma\in\{\pm1\}^d} \bbP\left(\Lambda_k\times\bbM \not\leftrightarrow F^{j,\sigma}_N\times\bbM \text{ in }\L_N\times\bbM\text{ in } G(\bbX)\right)\right)^{1/(d2^d)}\\
\nonumber&{}\le \bbP(\Lambda_k\times\bbM \not\leftrightarrow \partial_*\Lambda_N\times\bbM\text{ in } G(\bbX))^{1/(d2^d)}\leq e^{-1/(d 2^d \e)},
\end{align}
using Assumption~\ref{assumere}\ref{ass:stationary} for the first equality, Assumption~\ref{assumere}\ref{ass:symmetry} for the second one, the FKG inequality \eqref{eq:FKG} (indeed the involved events can be restated in the space $\cN(\bbW)$) for the second inequality and \eqref{eq:1arm:coarse} for the last one.

By a union bound \eqref{eq:sqrt:trick} gives
\begin{equation}
\label{eq:cup:Bi}\bbP \left(\bigcup_{i=1}^l \cB_i\right)\leq l e^{-1/(d 2^d \e)}\leq e^{-2c/\e},\end{equation}
since $l \leq 1/\e^2+1$ and $\varepsilon$ and $c$ are small enough.

Note that $\cE_i\setminus\cB_i\subset \{(x_i,m_i)\leftrightarrow\partial_*\Lambda_N\times\bbM\text{ in }G(\bbX\cup S)\}$ for all $i\in\{1,\dots,l\}$ (here we use the event $\cA(k,K,x_i)$). Thus, \eqref{eq:cup:Ei} and \eqref{eq:cup:Bi} give
\[\bbP\left(S\leftrightarrow \partial_*\Lambda_N\times\bbM\text{ in } G(\bbX\cup S)\right)\ge \bbP\left(\bigcup_{i=1}^l(\cE_i\setminus\cB_i)\right)\geq \bbP\left(\bigcup_{i=1}^l\cE_i\right)-\bbP\left(\bigcup_{i=1}^l \cB_i\right)\geq 1- 2e^{-2c/\varepsilon}.\]
Note that, since ${\ell}_*=1$ (recall  Section~\ref{subsec:notation} and Assumption~\ref{assumere}\ref{ass:finite:support} for the definition of $\ell_*$), in the event in the l.h.s.\ we can add that the path connecting $S$ with $\partial_*\Lambda_N\times\bbM$ lies inside $\L_N\times\bbM$. Hence, to conclude the proof, it remains to replace $\partial_*\L_N$ with $F_N$ above.
To this end, notice that the event $\{S\leftrightarrow F_N\times\bbM\text{ in }\L_N\times\bbM\text{ for }G( \bbX\cup S)\}$ is implied by   the intersection of the following three events:
 \[\{S\leftrightarrow \partial_*\Lambda_N\times \bbM\text{ in } G(\bbX\cup S)\}\,,\qquad  \{\Lambda_k\times\bbM \leftrightarrow F_N\times\bbM\text{ in }\L_N\times\bbM \text{ in } G( \bbX) \}\,,\qquad \cA(n,N),\]
where the event $\cA(n,N)$ refers  to $G(\bbX)$. Consider \eqref{eq:sqrt:trick}, but only starting from the third probability there, and observe that it remains true also when replacing  $F_N^{j(i), \s(i)}$ by $F_N$. Then, recalling also
 \eqref{eq:uniqueness:n}, we obtain the desired bound
\[\bbP(S\leftrightarrow F_N\times\bbM\text{ in }\L_N\times\bbM \text{ in }G(\bbX\cup S))\ge 1-2e^{-2c/\varepsilon}-e^{-1/(d2^d\varepsilon)}-e^{-1/\varepsilon}\ge 1-e^{-c/\varepsilon}.\qedhere\]
\end{proof}

For the next result recall that $\l'$ is a fixed density with $\l'>\l$ (see Section~\ref{subsec:notation}). This lemma is our starting point for dealing with the presence of bad marks in the model.
\begin{Lemma}[Seed to seed]
\label{lem:seed:to:seed} There exists an event 
$\cH\subset \O\times[0,1]^J$ 
such that 
\begin{enumerate}
\item \label{cond:proba}$P_{\lambda',\nu}\otimes Q(\cH)>1-\varepsilon$;
\item \label{cond:deterministic}if $(\o,\underline u )\in\cH$, then any path $(x_i,m_i)_{i=0}^r$ in $\Lambda_{6N}\times\bbM$ for $G(\o,\underline{u})$ such that $|x_r-x_0|_\infty\ge  \sqrt{N} -1$ contains a seed;
\item \label{cond:decreasing} the event $\cH$ is of the form $\cH:=\{(\o,\underline u)\in \O
\times [0,1]^J: G(\o,\underline u)\in \cH_*\}$ for a suitable set  $\cH_*$ of graphs with vertex set in $\bbR^d
\times \bbM$.  $\cH_*$ is  decreasing w.r.t.\ inclusion (i.e.\ if $\cG\in\cH_*$ and  $\cG\supset\cG'$, then $\cG'\in\cH_*$) and a graph on $\bbR^d\times\bbM$ belongs to $\cH_*$ if and only if its restriction  to $\L_{6N}\times \bbM$ belongs to $\cH_*$.
\end{enumerate}
\end{Lemma}

\begin{Remark}\label{rem:seed:to:seed}
Thanks to Item~\ref{cond:decreasing}, Lemma~\ref{lem:seed:to:seed} remains valid if, in Item \ref{cond:proba}, we replace $P_{\l',\nu}$ by $P_{\mathfrak{m},\nu}$,  $\mathfrak{m}$ being any measure on $\bbR^d$   dominated by $\l' {\cL} $.
\end{Remark}

\begin{proof}[Proof of Lemma~\ref{lem:seed:to:seed}] We set $P:=P_{\l',\nu}\otimes Q$.
Recall the definition \eqref{set_good_marks} of good marks.  
 Below, just to simplify the notation but w.l.o.g., we assume that $\sqrt{n}$ is integer. 
 
For $z\in n\bbZ^d$, we say that the box $\Lambda_n(z)$ is \emph{bad}, if there is a  path in $G(\o,\underline{u})$ with vertices  in  $(\Lambda_n(z)\cap\Lambda_{6N})\times \bbM_g^c$ of  length larger than $\sqrt n$; and \emph{good} otherwise. Note that boxes $\Lambda_n(z)$ which do not intersect $\Lambda_{6N}$ are automatically good.

Let us  bound from above  the probability that $\L_n(z)$ is bad by a known
  branching process comparison.
When sampling $\o$ with law $P_{\l' ,\nu}$,  the set $\{x \in \bbR^d\,:\, (x,m_x)\in \o\,,\;  m_x \in \bbM_g^c\}$ is distributed as a  homogeneous PPP $\xi$ on $\bbR^d$ with intensity $\rho:=\l' \nu(\bbM_g^c)$.
Consider now the RCM built on  $\xi$  with connection function $\varphi'(x,y):= \1( |x-y|_\infty\leq 1
)$ (i.e.~the  Poisson--Boolean model with radius $1/2$). Then,  considering $\xi$ conditioned to contain the origin (equivalently, taking the RCM on $\xi\cup\{0\}$),   the number $N_k$ of points with graph distance $k$ from the the origin is stochastically dominated by the number $\tilde N_k$ of points  in the $k$-th generation of a Bienaym\'e--Galton--Watson branching process with expected offspring equal to $2^d \rho$  (see e.g.\ the proof of  \cite{Meester96}*{Theorem~6.1}).  
Hence, by the Markov inequality,
\[\bbP(N_k\ge 1)\le \bbE[\tilde N_k]=(2^d\rho)^k.\]
Considering the RCM with connection function $\varphi'$ on  the  homogeneous PPP $\xi$
with intensity $\rho$,
we define $A_x$ as the event that this RCM  has at least one point at graph distance $\sqrt{n}$ from $x$ when $x$ is a vertex.
Using  that ${\ell}_*=1$ in Assumption~\ref{assumere}\ref{ass:finite:support} we can therefore bound\footnote{We take $\bbP$  defined on the space  $\cN(\bbR^d)$ of locally finite subsets of $\bbR^d$  and $\xi$ the identity map.}
\be\label{branching}
\begin{split}
P (\L_n(z)\text{ is bad})&\leq P(\L_n\text{ is bad})\leq \bbP \Big(
\bigcup_{x\in \xi\cap \L_n}  A_x \Big)\\
& \leq \bbE\Big[ \sum_{x\in \xi\cap \L_n}  \1_{A_x}\Big]  = \rho \int_{\L_n} \bbP(\xi\cup\{x\}\in A_x)\,\md x  \leq  (2n)^d 
(2^d \rho)^{\sqrt{n}}\rho
\end{split}
\en
(the above identity follows from Mecke equation, see \cite{Last18}*{Theorem~4.1}). By Lemma~\ref{lemma_scale}, we have $2^d\rho\le 2^d\lambda' c<1$ (by choosing $c$ small), so \eqref{branching} can be made arbitrarily small by taking $n$ large.

Let $\cH$ be the event such that its complement $\cH^c$ is the following: there are at least  
$\sqrt[3] N$ distinct  bad boxes $\L_n(y_1),\L_n(y_2),\dots, \L_n(y_{\sqrt[3]N})$ with successive centers in $n\bbZ^d$  at $\ell^\infty$-distance at most $4n$ (to simplify the notation we assume $\sqrt[3]N$ to be integer). 
The event $\cH$ clearly satisfies \ref{cond:decreasing} as $\cH^c$ is described in terms of $G(\o,\underline u)$ and is increasing w.r.t.\ graph inclusion.
 
Let us now prove \ref{cond:proba}. Since good boxes form a 1-dependent process, the Liggett--Schonmann--Stacey theorem \cite{Liggett97}  (also see \cite{Grimmett99}*{Theorem~(7.65)}) implies the following. For $p\in (0,1)$, by taking the rightmost term in \eqref{branching} small enough (that is, $n$ large enough), the random field $(T_y)_{y\in\bbZ^d}$ with   $T_y:=\1(\L_n(ny) \text{ is good})$ stochastically dominates a Bernoulli site percolation with parameter $p$. 
Due to the this stochastic domination  and since bad boxes have to intersect $\L_{6N}$, we have  $P(\cH)\ge 1-\varepsilon$, taking $p$ close to 1 (and therefore $n$ large) and afterwards $N$ large enough.  This follows easily since for the above Bernoulli site percolation  the probability to have a 
path of length $\sqrt[3] N$ of closed points in $\L_{\lceil 6N/n\rceil +1}$ such that consecutive points have $\ell^\infty$--distance at most $4$  is upper bounded by  $(12N+5)^d\left(9^d 
(1-p)\right)^{\sqrt[3] N}$.

We next prove \ref{cond:deterministic}.
We first observe that the Voronoi tessellation associated to $n \bbZ^d$ is given by the boxes $\L_{n/2}(z)$ with $z\in n\bbZ^d$. In particular, for each $x\in \bbR^d$ we can find $z\in n \bbZ^d$ such that $x \in \L_{n/2}(z)$. Fix $(\omega,\underline{u}) \in\cH$ such that there exists a path $\gamma=(x_i,m_i)_{i=0}^r$ as in the statement of \ref{cond:deterministic}.  Let $z_i\in n \bbZ^d$ be such that $x_i\in \L_{n/2} (z_i)$. Recall that $|x_{i+1}-x_i|_\infty \leq 1$ for any consecutive points in $\g$ (as ${\ell}_*=1$) and $|x_r-x_0|_\infty\geq \sqrt{N}-1$. Then we  prune the finite sequence $\L_n (z_1),
\L_n (z_2),\dots, \L_n(z_r)$, by loop-erasing and by removing intersecting boxes,  and we  extract distinct points $y_0,y_1,\dots,y_s$ among $z_0,z_1,\dots,z_r$ such that $|y_{j+1}-y_j|_\infty =2n$ for each $j=0,1,\dots, s-1$, where $y_0:=z_0$, and $|y_{s}- z_r|_{\infty}\leq 2n$. Note that necessarily $s\geq (\sqrt{N}-1-2n) /(2n)$. For each $j=1,2,\dots, s$ we note that   the path $\gamma$ goes  from $(x_0,m_0)$ into  $ \Lambda_{n/2}(y_j)\times\bbM$, hence  $\gamma$ needs to give rise to a path $\gamma_j$  in $ \Lambda_n(y_j)\times\bbM$ of $\ell^\infty$-diameter at least $n/2-1$.  We claim that some  $\g_j$ contains a seed.

Assume by contrapositive  that, for all $j=1,2,\dots,s $, $\gamma_j$ does not contain a seed. Fix $j$. We aim to prove  that $\L_n(y_j)$ is bad.  Let us call $Y_j$ a maximal subset of vertices of  $\g_j$ in $\Lambda_n(y_j)\times\bbM_g$ with mutual $\ell^\infty$-distance at least $2K$.  Since there is no seed, $\sharp Y_j\leq 1/\e^{2}$. By maximality of $Y_j$, any other vertex of $\g_j$ in $\Lambda_n(y_j)\times\bbM_g$ has $\ell^\infty$-distance from $Y_j$ at most $2K$. Hence the set $\tilde Y_j$ of vertices of $\g_j $ with marks in $\bbM_g$ is contained in the union of $1/\varepsilon^2$ $\ell^\infty$-balls of radius $2K$. Call $D$ the maximal diameter of a sub-path of $\gamma_j$ in $\Lambda_{n}(y_j)\times\bbM_g^c$. Since the diameter of $\gamma_j$ is at most the diameters of the above balls plus the lengths of paths between them in some linear order, we get
\[n/2-1\le \mathrm{diam}(\gamma_j)\le (4K+D+1)(1/\varepsilon^2+1)
\,.\]
In view of \eqref{eq:scales}, we get $D>\sqrt n$.

Consequently, the box $\Lambda_n(y_j)$ is bad, and this holds for any $j$.  Since we can fix the scales so that $(\sqrt{N}-1-2n)/(2n)>\sqrt[3] N$, the boxes $\L_n(y_i)$  contradict the occurrence of the event $\cH$, thus concluding the proof. 
\end{proof}

\section{Exploration process to find LR crossings}\label{sec_tanemura}
Just in this section,   given    $A\subset \bbZ^2$, we set  $\D A:= \{ y\in \bbZ^2 \setminus A\,:\, |x-y|=1\text{ for some } x\in A\}$.

In this section we review  and extend the exploration process introduced in \cite{Tanemura93}*{Section~4} to verify that  a random field on $\bbZ^2$ has enough LR crossings of a given box. The extension is due to the need in our applications that these LR crossings also lie inside the infinite cluster. In our extension  we have kept as much as possible the same notation as \cite{Tanemura93}*{Section~4} (the first part is also close  to  \cite{Faggionato21}*{Section~4}).

 Before giving the formal details, let us describe the argument informally in a simpler geometric context. Suppose that we seek to find the maximal number of vertex-disjoint LR paths crossing a box in a Bernoulli site percolation configuration. To do so, we explore the configuration, by finding the lowest crossing from the lowest point on the left boundary of the box, if it exists. Then, we explore the part of the box above the already explored path similarly starting from the second lowest point and so on. Lemma~\ref{lemma_pistacchio} below states that, if the success probability of each site being open is large enough, then this procedure yields many LR crossings with high probability. 
 A notable complication in our setting is that we need to ensure that the disjoint crossings belong to the infinite cluster, which makes exploration arguments a bit more delicate and which requires new specific estimates (cf.~Section~\ref{sec_prob_bounds}).

   \subsection{Order \texorpdfstring{$\prec$}{<} on \texorpdfstring{$\D \{x_1, \dots, x_n\}$}{Delta\{x1,...,xn\}}}
   \label{total_order}
Let $(x_1, x_2, \dots, x_n)$  be a string of points in $\bbZ^2$, such that, for any integer $k$ with  $2\leq k \leq n$, $x_k\in\Delta\{x_1,\dots,x_{k-1}\}$.  
  We first introduce a total order $\prec_k$ on the points neighbouring $x_k$ for any $k=1,2,\dots, n$. To this aim we 
let    $a(k):= \max \{j: 0\leq j< k \text{ and } |x_k-x_j|=1\} $, where $x_0=x_1-e_1$.  Then we $\prec_k$-order the points $y\in\bbZ^2$  adjacent to $x_k$ starting from $x_{a(k)}$ (minimal point) and moving clockwise (to  have increasing points). 

   The order $\prec$ on $\D  \{x_1, \dots, x_n\}$  is obtained as follows (we go from the largest  element to the first one). The largest elements are the points in $\D  \{x_1, \dots, x_n\}$   neighboring $x_n$ (if any), ordered according to $\succ_n$.
The next elements, in decreasing order, are the points in  $\D  \{x_1, \dots, x_n\}$  neighboring $x_{n-1}$ but not $x_n$ (if any), ordered according to $\succ_{n-1}$. As so on, in the sense that in the generic step one has to consider the points in   $\D  \{x_1, \dots, x_n\}$  neighboring $x_k$ but not $x_{k+1}, \dots, x_n$ (if any), ordered according to $\succ_k$.
   
   The above order is the one introduced in \cite{Tanemura93}*{Section~4} with a relevant exception. Indeed in \cite{Tanemura93} the author supposes that $(x_1,x_2,\dots,x_n)$ is a path  in $\bbZ^2$ but then refers to the ordering  also when dealing with more general strings, as the above ones. If $(x_1,x_2,\dots,x_n)$ is a path,  then $a(k)=k-1$ for all $k=1,2,\dots, n$  and the above definition  of $\prec$ coincides 
 with the order on $\D\{x_1,x_2,\dots,x_n\}$ introduced in \cite{Tanemura93}*{Section~4}.

   \subsection{Exploration process}
   \label{subsec:exploration}
   Given positive integers $L,M$, we consider the domain 
   \be\label{scooby}
   \Lambda_L'=\Big([0,M-1]^2\cup([M,M+L]\times [-L,M+L])\Big)\cap \bbZ^2\,,
   \en
with  \emph{left boundary} $\partial_l\Lambda_L'=\{0\}\times\{0,\dots,M-1\}$ and \emph{right boundary} 
\[\partial_r\Lambda_L'=\Big(\{M+L\}\times\{-L,\dots,M+L\}\Big) \cup\Big(\{M,\dots,M+L\}\times\{-L,M+L\} \Big).\] We denote $x_1^s=(0,s)$ for $s\in\{0,\dots,M-1\}$. Note that in the above notation $\Lambda_L'$, $\partial_l\Lambda_L'$ and  
$\partial_r\Lambda_L'$  the integer $M$ is kept implicit.
Suppose that on some probability space, with probability measure $\bbQ_{L}$ the following events are defined: $\{x_1^s:=(0,s)\text{ is occupied}\}$  for $s=0,1,\dots, M-1$, $\{y\text{ is linked to } x\}$ for $|x-y|=1$ and $x,y\in\Lambda'_L$.   
 
Having these events, we  can construct the random sets $C^s_j= ( E^s_j, F^s_j)$ of sites explored from the $s$-th starting point, which are reached or not respectively, with $s\in\{0,1,\dots,M-1\}$ and $j\in\{1,2,\dots, \sharp\Lambda_L'
\}$,  in the following order. We first define $C^0_1$, $C^1_1$, $\dots$, $C^{M-1}_1$, then iteratively
\be\label{cannolo42}
C^0_2, \;C^0_3, \dots,\; C^0_{\sharp\Lambda_L'},\; C^1_2,\; C^1_3,\dots,\; C^1_{\sharp\Lambda_L'},\dots,
\; C^{M-1}_2,\; C^{M -1}_3, \dots,\; C^{M-1}_{\sharp\Lambda_L'}\,.
\en

 The construction will ensure the following properties:
 \begin{itemize}
  \item $E^s_j\subset\L_L'$ and $F^s_j\subset\L_L'$ are disjoint sets;
   \item  there exists $J^s \in \{1,2\dots,\sharp\Lambda_L'\}$ such that, for $j< J^s$, 
 $C^s_{j+1}=( E^s_{j+1}, F^s_{j+1})$ is obtained from $C^s_j=( E^s_j, F^s_j)$   by adding exactly one point (called $x^s_{j+1}$) either to $E^s_j$ or to $F^s_j$ and $x^s_{j+1}\in \D E^s_j$ if added to $E^s_j$; while, for $j \geq J^s$, $C^s_{j+1}=( E^s_{j+1}, F^s_{j+1})$ equals $C_j^s=( E^s_j, F^s_j)$.
 \end{itemize}
 
 We start with the first block $C^0_1$, $C^1_1$, $\dots$, $C^{M-1}_1$ with $j=1$. Recall that $x^s_1:=(0,s)$. For all  $s=0,1,\dots, M-1$ we  build $C^s_1$ as follows:
 \begin{equation}\label{eq:pasqua} 
\begin{cases}  \text{if  $x^s_1$ is occupied, then 
$C^s_1 =  \left( E^s_1, F^s_1 \right):= (\{x^s_1\},\emptyset)$}\,,\\
\text{if   $x^s_1$ is not occupied, then 
$C^s_1 =  \left( E^s_1, F^s_1 \right):= (\emptyset,\{x^s_1\})$}\,.
\end{cases}
\end{equation}

We now define iteratively   the sets in \eqref{cannolo42}   as follows.

If $x_1^s$ is not occupied (i.e.~$E^s_1=\emptyset$), then  we set   $J^s:=1$, thus implying that $C_1^s=C_{2}^s=\cdots= C_{\sharp\Lambda_L'}^s$. 

If  $x_1^s$ is occupied (i.e.~$E^s_1\not =\emptyset$), then we proceed as follows. Suppose that we have defined all the sets  preceding $C^s_{j+1}$  in the above string \eqref{cannolo42} {(i.e.\ up to $C^s_j$), that we have not assigned to  $J^s$ any value yet  and that we want to define $C^s_{j+1}$.  
  We call $W^s_j$ the points of $\L_L'$ involved in the construction up to this moment, i.e. \[
W^s_j = \{x_1^k:0\leq k \leq M-1\} \cup \{ x^{s'}_r :  0 \leq s'< s, \, 1< r \leq \sharp\Lambda_L'\}
\cup \{ x^{s}_r : 1< r \leq j\}
\,.
\]
We recall that it must be $E^s_0 \subset E^s_1\subset \cdots \subset E^s_j$ and at each inclusion either the two sets are equal or the second one is obtained from the first one by adding exactly one point adjacent to the previous one.
Hence, we can order the elements of $E^s_j$ according to the order with which the points have been added. This order on $E^s_j$ induces an order $\prec $ on the boundary $\D E^s_j$ as described in Section~\ref{total_order}. We call  $\cP^s_j$ the following property:  $E^s_j$ is disjoint from  the right boundary $\partial_r\Lambda_L'$ and   $(\L_L'\cap  \D{E}^s_j)  \setminus W^s_j \not = \emptyset$.   If property $\cP^s_j$ is satisfied, then  we define $x^s_{j+1} $  as the largest element of $(\L_L' \cap \D{E}^s_j)  \setminus W^s_j $  w.r.t.\ $\prec$. Let  $k$  be  the largest integer such that  $x_k^s\in E^s_j$ and $|x^s_{j+1} -x^s_k |=1$.  We decide where to add $x^s_{j+1}$ as follows:
\be\label{reveal}
\begin{cases}
\text{if $x^s_{j+1}$ is linked to $x^s_k$, then  $C^s_{j+1} :=\big( E^s_j \cup \{ x^s_{j+1}\}, F^s_j \big)$,}\\
\text{if $x^s_{j+1}$ is not linked to $x^s_k$, then $C^s_{j+1} :=\left( E^s_j, F^s_j \cup \{ x^s_{j+1}\} \right)$\,.}
\end{cases}
\en

On the other hand, if property $\cP^s_j$ is not verified, then we set $J^s:=j$ and $C_j^s=C_{j+1}^s=\cdots= C_{\sharp\Lambda_L'}^s$.

Let us now consider the random  graph $\mathfrak{G}_L$ with vertex set $\Lambda'_L\setminus\bigcup_{s=0}^{M-1} F^s_{\sharp\Lambda'_L}$ and with edges between vertices at Euclidean distance $1$. According to \cite{Tanemura93}*{Section~4}, the following statement is easy to see and, therefore, we leave its proof to the reader.

\begin{Claim}[Correctness]\label{claim:tanemura}
The maximal number $N_{L}$ of vertex-disjoint paths in the random graph $\mathfrak{G}_L$ from $\partial_l\Lambda_L'$ to $\partial_r\Lambda_L'$ in $\Lambda_L'$ is equal to the number of $s\in\{0,\dots,M-1\}$ such that $E^s_{\sharp\Lambda_L'}\cap \partial_r\Lambda_L'\neq\varnothing$.
\end{Claim}

\subsection{Probabilistic bounds}\label{sec_prob_bounds}
Before studying crossings under the measure $\bbQ_{L}$, we need a few preliminaries on the Bernoulli site percolation measure $\bbP_p$ with parameter $p\in[0,1]$ on $\{0,1\}^{\bbZ^2}$. We denote $\pc(2)=\inf\{p>0:\bbP_p(0\leftrightarrow\infty)>0\}$.

\begin{Lemma}
\label{lem:crossing}
    Let $p_0>\pc(2)$. Then there exists $c_0>0$ such that, for all $M,L\ge 1$ we have
    \[\bbP_{p_0}(\partial_l\Lambda'_L\leftrightarrow\partial_r\Lambda'_L\text{ in }\Lambda'_{L})\ge 1-e^{-c_0M}.\]
\end{Lemma}
\begin{proof}
We opt for a short, albeit not very robust proof specific to the present setting. By duality (cf.~\cite{Kesten82}*{Proposition~2.2}), we need to prove that, for all $p'<1-\pc(2)$, in the square-diagonal lattice (which is the matching lattice of the square lattice, cf.~\cite{Kesten82}*{Chapter~2}), the $\bbP_{p'}$-probability of connecting the remaining two pieces of the boundary of $\Lambda'_L$ is at most $e^{-\k M}$ for some $\k >0$. This follows by a union bound, exponential decay in the subcritical phase \cite{Aizenman87} and the fact that the critical probability of the square-diagonal lattice is $1-\pc(2)$ (see \cite{Russo81} or \cite{Kesten82}*{Chapter 3}).
\end{proof}
For the next result we recall the definition of $I_l(A)$,  the  interior of depth $l$  of an event $A$ in $\{0,1\}^m$: 
$I_l(A)$ is given by the configurations $\s$ in $A$ such that, by changing at most $l$ entries of $\s$, the new configuration is still in $A$ (cf.~\cite{Aizenman83}, \cite{Grimmett99}*{Section~2.6}).

\begin{Lemma}[\cite{Aizenman83}*{Lemma~4.2}]
\label{lem:modification:distance}
    Let $m,l$ be positive integers, and let $p>\d>0$. Let $A$ be an increasing event on $\{0,1\}^m$ equipped with the product Bernoulli measure $\bbP_p$ of parameter $p$. Then,
    \[\bbP_p\left(I_l(A)\right)\ge 1-(1-\bbP_{p-\d}(A))/{\d}^l.
    \]
\end{Lemma}
The next lemma is an application of Lemmas~\ref{lem:crossing} and~\ref{lem:modification:distance} following the proof of \cite{Chayes86}*{Lemma~3.3}.
\begin{Lemma}\label{lemma_pistacchio}
Let $p>\pc(2)$. Denote by $N_L'$ the number of vertex-disjoint open site percolation crossings from $\partial_l\Lambda'_L$ to $\partial_r\Lambda'_L$ in $\Lambda'_L$. Then there exist $\k,\k'>0$ such that for all positive $L,M$, we have $\bbP_p(N_L'\geq \k  M) \geq 1- e^{-\k' M}$.
\end{Lemma}
\begin{proof}[Proof of Lemma~\ref{lemma_pistacchio}]
Setting $p_0:=(p+\pc(2))/2$, let $c_0>0$ be as in Lemma~\ref{lem:crossing}.
Let $A:=\{\partial_l\Lambda_L'\leftrightarrow\partial_r\Lambda_L'\text{ in }\Lambda_L'\}$. By Menger's theorem, for $\k>0$ to be chosen later, we have $\{N_L' >\k M\}=I_{ \k M}(A)$. By Lemma~\ref{lem:modification:distance}, applied with $\d:=p-p_0$, and then Lemma~\ref{lem:crossing}, we have
\[\bbP_p(N_L'\ge \k M)\ge 1-(1-\bbP_{p-\d}(A))/\d^{\k M}\ge 1-e^{-c_0M}\d^{-\k M}=1-e^{-\k 'M}\]
for $\k':=c_0+\k\ln \d$, which can be made positive by choosing $\k>0$ small enough.
\end{proof}

In order to pass the result of Lemma~\ref{lemma_pistacchio} to $\bbQ_L$, we introduce the following notion.

\begin{Definition}[Domination]
\label{def:domination}
Let $p\in[0,1]$. We say that a probability measure $\bbQ_{L}$ as in Section~\ref{subsec:exploration} \emph{dominates} $p$, if:
\begin{itemize}
    \item going from $s=0$ to $s=M-1$ the probability that $x^s_1$ is occupied, conditioned to the revealed status (i.e.\ occupied or not occupied) of $x_1^{s'}$ with $s'<s$, is at least $p$,
    \item every time we check which one of the two cases in \eqref{reveal} is valid, the probability that $x^s_{j+1}$ is linked to $x_k^s$,  conditioned to the accumulated information (i.e.\ to the knowledge of $C^a_1$ with $0\leq a\leq M-1$, of $C_r^{s'}$ with $0 \leq s' <s$ and $1<r\leq \sharp\Lambda_L'$ and of $C_r^s$ with $1<r\leq j$), is at least $p$.
\end{itemize}
\end{Definition}

\begin{Corollary}
\label{cor:domination}    Assume that the measure $\bbQ_L$ dominates some $p>\pc(2)$. Recall $N_L$ from Claim~\ref{claim:tanemura}. Then there exist $\k ,\k '>0$ such that for all positive $M,L$, we have 
    \[\bbQ_L(N_L\ge \k M)\ge 1-e^{- \k 'M}.\]
\end{Corollary}
\begin{proof}
    In view of Definition~\ref{def:domination}, there is a monotone coupling of the random graph $\mathfrak G_L$ under $\bbQ_L$ with the open site-percolation clusters intersecting $\partial_l\Lambda'_L$ under $\bbP_p$ (cf.~e.g.~\cite{Grimmett90}*{Lemma~1}). Thanks to Claim~\ref{claim:tanemura} applied to $\bbP_p$, under this coupling we have $N_L\ge N_L'$ a.s. The result then follows from Lemma~\ref{lemma_pistacchio}.
\end{proof}

\section{Proof of Theorem~\ref{teo_crossings}}\label{proof_teo_crossings}
In this section we suppose Assumption~\ref{assumere} to be true and aim to prove Theorem ~\ref{teo_crossings} for  $\l'>\l$ (instead of $\rho>\l$), by combining the results and tools of the previous Sections  \ref{sec_preliminaries}, \ref{sec_connections} and  \ref{sec_tanemura}. 

We take the scales $N\gg n\gg K\gg k\gg1/\varepsilon\gg 1/c\gg 1$ as in \eqref{eq:scales} in Section~\ref{sec_connections}, further chosen below depending on a fixed $p_0\in(\pc(2),1)$.  We also consider the slab $\Sigma:=\bigcup_{(x,y)\in \bbZ^2} \L_{6N}(12N(x,y,0,\dots,0))$.

By standard arguments (cf.~e.g.~\cite{Chayes86}*{Lemma 3.3}, \cite{Faggionato21}*{Proposition 3.6}), in order to prove Theorem~\ref{teo_crossings}, it is enough to show the following: 
 for any $\l'>\l$,  there exist $c_1,c_2>0$ such that, for all $\ell$ large enough, we have
\be\label{2d_bound}
P_{\lambda',\nu} \otimes Q\left( \cM_{\ell}   \geq  c_1 \ell\right)\geq 1- \exp\left(- c_2 \,\ell \right)\,,
\en
where $\cM_{\ell}$ denotes the maximal number of vertex-disjoint LR crossings of $\L_{\ell}$ contained in the slab $\Sigma$  and which can be extended to paths to infinity in $\Sigma$ (Assumptions~\ref{assumere}\ref{ass:irreducible} and \ref{ass:supercritical} imply that $P_{\l',\nu}$--a.s.\ there is a unique infinite cluster, these paths must lie in the infinite cluster).

To prove \eqref{2d_bound}, we combine the exploration process of Section \ref{sec_tanemura} and the connection estimates of Section~\ref{sec_connections}. To simplify the presentation, but without loss of generality, we suppose that $2\ell$ is a multiple of $12 N$ and write $2\ell=12 N M$ with $M$ positive integer. 

On a common probability space with law $P$, we consider independent PPPs  $\xi,  \xi_1,\xi_2,\dots, \xi_{400d}$ on 
 $\bbR^d\times\bbM$ with intensity measures $\l\cL \otimes\nu$ for $\xi$ and $\frac{\l'-\l}{400d}\cL \otimes\nu$ for $\xi_i$.  
 We  define
$ \xi_{( i,j]} :=\bigcup_{k=i+1}^j \xi_k$.
  Note that their superposition $ \xi^*:=\xi\cup \xi_{(0,400d]}$   has law $P_{\l',\nu}$.
  
Given a graph $G=(V,E)$ in $\bbR^d\times\bbM$ and given $A\subset \bbR^d\times \bbM$, we write $G|A$ for the restriction of $G$ to $A$, which is the graph with vertex set $V\cap A$ and edge set $\{\{t,t'\}\in E:t,t'\in A\}$.   We fix a deterministic seed $S_0\subset \L^{(0)}\times\bbM_g$ where $\L^{(0)}:=[-13N/2,-6N]\times[-4N,4N]^{d-1}$ (recall \eqref{set_good_marks}, Definition~\ref{def:seed} and see Figure~\ref{fig:renorm}).

\begin{figure}
\centering
\begin{tikzpicture}[scale=0.4]
\draw (-18,-6) rectangle (6,6);
\draw (-6,-6) -- (-6,6);
\draw (-13/2,-4) rectangle (-6,4);
\draw (-4,-6) rectangle (4,-11/2);
\draw (11/2,-4) rectangle (6,4);
\draw (4,6) rectangle (-4,11/2);
\draw (-12,-6) node[below]{$\Lambda_{6N}(-12Ne_1)$};
\draw (0,-6) node[below]{$\Lambda_{6N}$};
\draw (12,-6) node[below]{$\Lambda_{6N}(12Ne_1)$};
\draw (-13/2,2) node[left]{$\Lambda^{(0)}$};
\draw (2,-11/2) node[above]{$\Lambda^{(1)}$};
\draw (11/2,2) node[left]{$\Lambda^{(2)}$};
\draw (2,11/2) node[below]{$\Lambda^{(3)}$};
\draw (10,-11/2) node[above]{$\Lambda_2^{(1)}$};
\draw (12+11/2,-2) node[left]{$\Lambda_2^{(2)}$};
\draw (8.5,11/2) node[below]{$\Lambda_2^{(3)}$};
\fill (-25/4,-1) node[left]{$S_0$}circle (0.15);

\draw (-25/4,-1) .. controls (-10,-3) and (-7,-8) .. (-3,-1);
\draw (-4,-1) .. controls (-4,-3) and (-2,-5) .. (-1,-5.6);
\fill (-1,-5.6) node[above]{$\cS^{(1)}$}circle (0.15);
\draw (-25/4,-1) .. controls (-8,1) and (-7,4) .. (2,3);
\draw (-4,1) .. controls (-4,1) and (-3,3) .. (-1,5.6);
\fill (-1,5.6) node[below]{$\cS^{(3)}$}circle (0.15);
\draw (1,4) .. controls (-2,0) and (3,4) .. (5.8,-2);
\fill (5.8,-2) node[left]{$\cS^{(2)}$}circle (0.15);
\draw (5,5) .. controls (3,2) and (0,3) .. (1,-1);
\draw (-7,-4) .. controls (-12,-5) and (-16,8) .. (-14,-1);
\draw (-11,-5) .. controls (-8,3) and (-14,7) .. (-9,-1);

\draw (6,-6) rectangle (18,6);
\draw (-4,-6) rectangle (4,-11/2);
\draw (8,-6) rectangle (16,-5.5);
\draw (35/2,-4) rectangle (18,4);
\draw (8,6) rectangle (16,5.5);

\draw[dashed] (-5,-5) .. controls (1,-1) and (8,-10) .. (17.8,2);
\fill (17.8,2) node[left]{$\cS^{(2)}_2$}circle (0.15);
\draw[dashed] (11,-1) .. controls (10,-3) .. (15,-5.7);
\fill (15,-5.7) node[above]{$\cS^{(1)}_2$}circle (0.15);
\draw[dashed] (3,-1) .. controls (7,3) .. (13,5.75);
\fill (13,5.75) node[below]{$\cS^{(3)}_2$}circle (0.15);
\draw [decorate,decoration={brace,amplitude=4pt},xshift=0pt,yshift=0pt]
(6,-7.2) -- (-18,-7.2)node[midway,below,yshift=-0.2]{$\L$};
\draw [decorate,decoration={brace,amplitude=4pt},xshift=0pt,yshift=0pt]
(-18,6.2) -- (18,6.2)node[midway,above,yshift=0.2]{$\L'$};
\end{tikzpicture}
\caption{\label{fig:renorm}
Illustration of Definitions~\ref{def:occupied} and \ref{def:link}. The last $d-2$ dimensions and marks are suppressed. The thin rectangles represent the boxes 
$\Lambda^{(0)},\Lambda^{(1)},\Lambda^{(2)},\Lambda^{(3)},\Lambda_2^{(1)},\Lambda_2^{(2)},\Lambda_2^{(3)}$, where the deterministic seed $S_0$ and the random ones $\cS^{(1)},\cS^{(2)},\cS^{(3)},\cS_2^{(1)},\cS_2^{(2)},\cS_2^{(3)}$ are located. The event that the renormalised site $(0,0)$ is occupied corresponds to the existence of the seeds $\cS^{(1)},\cS^{(2)},\cS^{(3)}$ and to $S_0$ being connected to them inside the box $\L$, as depicted in solid lines. Only the quarter $\xi_{(0,100d]}$ of the total sprinkling in this box is used at this point. The renormalised site $(0,0)$ being linked to $(1,0)$ corresponds to having the other three seeds and the dashed connections which may use a further quarter of the sprinkling in the middle box $\Lambda_{6N}$, as well as the first quarter of the sprinkling in the rightmost box $\Lambda_{6N}(12Ne_1)$.}
\end{figure}
The next crucial definition is illustrated in Figure~\ref{fig:renorm}.
\begin{Definition}[Occupation]
\label{def:occupied}
Let $\Lambda:=\L_{6N}(-12N e_1)\cup\L_{6N}=[-18N,6N]\times[-6N,6N]^{d-1}$ and $\Lambda^{(1)}:=[-4N,4N]\times [-6N,-11N/2]\times[-4N,4N]^{d-2}$, $\Lambda^{(2)}:=[11N/2,6N]\times[-4N,4N]^{d-1}$, $ \Lambda^{(3)}:=-\Lambda^{(1)}$.
We say that the site $(0,0)$ is \emph{occupied} if, for each $i\in\{1,2,3\}$, there exists a seed \[\cS^{(i)}\subset \Big\{t\in \Lambda^{(i)}\times\bbM_g: t\leftrightarrow S_0 \text{ in  }G\big(S_0\cup  \xi \cup \xi_{(0, 100d]}\big) \text{ in } \Lambda\times\bbM\Big\}.\]
\end{Definition}

In order to satisfy Definition~\ref{def:domination}, we prove the following statement where $p_0\in (\pc(2), 1) $ has been fixed once and for all (recall that $\pc(2)$ is the critical probability of Bernoulli site percolation on $\bbZ^2$).
\begin{Lemma}[Occupation probability]
\label{passo0} Choosing the scales in \eqref{eq:scales} suitably, it holds that $P(\text{$(0,0)$ is occupied})>p_0$.
\end{Lemma}
\begin{proof}
Recall \eqref{eq:def:F}.  For $\sigma\in\{-1,1\}^d$ and $i\in\{1,\dots,d\}$, let 
\[\tilde F^{i,\sigma}_N:=\{\sigma_iN\}^{\{i\}}\times\prod_{j\in\{1,\dots,d\}\setminus\{i\}}(\sigma_j[0,N]),\]
i.e.~$\tilde F^{i,\sigma}_N$ is given by the sites $x=(x_1,x_2,\dots, x_d)\in F_N^{i,\sigma}$ with $x_i=\s_i N$. We omit $i,\sigma$, if $i=1$ and $\sigma=(1,\dots,1)$. Let $z_0\in\bbR^d$ be such that $S_0\subset\L_{n-K}(z_0)\times\bbM_g\subset \L^{(0)} \times\bbM_g$ and, without loss of generality, assume that $z_0\in\bbR\times(-\infty,0]^{d-1}$ (in the general case one just has to replace in what follows $F_N$ by $F_N^{1,\s}$ where $\s\in\{-1,1\}^d$ satisfies $\s_1=1$ and $\s_j z_0(j)\le
0$ for $1<j\leq d$ with $z_0=(z_0(1),\dots,z_0(d))$). Let $\cD$ be the event that $G(\xi^*)|(\Lambda_{6N}\times\bbM)$ belongs to the set of graphs $\cH_*$ from Lemma~\ref{lem:seed:to:seed}. We first seek to connect $S_0$ to a seed close to $(z_0+{\tilde F}_N)\times\bbM$. By Lemma~\ref{lemma_crescita}, we get
    \begin{equation}\label{russell1}
    P \Big(S_0\leftrightarrow  (z_0+F_N)\times\bbM \text{ in  } (z_0+\Lambda_N)\times \bbM 
    \text{ in } G( S_0 \cup \xi'))\Big )\ge 1-e^{-c/\varepsilon},
\end{equation}
where $\xi':=\xi\cap (\Lambda\times\bbM)$. We call $E_1$ the event in the l.h.s.\ of \eqref{russell1} and we set 
\[ \cC:=S_0\cup\Big\{t\in \xi'\,:\, t\leftrightarrow S_0 \text{ in }  G( S_0 \cup \xi')\Big\}\,.
\]

By Lemma \ref{lem:seed:to:seed}\ref{cond:proba} we can estimate
$P( E_1\cap \cD)\geq P(E_1)- P(\cD^c) \geq 1-e^{-c/\varepsilon}-\e$.
The event $E_1\cap \cD$ and the decreasing property of $\cH_*$ in Lemma \ref{lem:seed:to:seed}\ref{cond:decreasing}  imply that the graph  $G(S_0 \cup \xi')|(\cC\cap(\Lambda_{6N}\times \bbM))$  belongs to $\cH_*$ (using that $S_0\cap(\Lambda_{6N}\times\bbM)=\varnothing$). Moreover, $E_1$  implies that there is a path of length at least $\sqrt{N}-1$ inside $G(S_0 \cup \xi')|(\cC\cap(\Lambda_{6N}\times\bbM))$ which intersects $z_0+F_N$. Hence, by Lemma~\ref{lem:seed:to:seed}\ref{cond:deterministic}, $E_1\cap \cD$  implies the following event $E_2$: $\cC$ contains a seed 

\[
\cS_1\subset  \L_{n-K}(z_1) \times\bbM_g  \subset (z_0+\tilde F_N+\L_{\sqrt{N}})\cap \bbM_g
\]
(using the bounds on the position of $S_0$, which imply that $z_0+F_N$ is well inside $\Lambda_{6N}$).
Due to the above considerations we conclude that 
\be\label{suono1}
P(E_1\cap E_2)\geq 1-e^{-c/\varepsilon}-\e=:\a_1\,.
\en

Note that the event $E_1\cap E_2$ is measurable with respect to the $\s$-algebra $\cT$ generated by $G(S_0\cup\xi')|\cC$. Let us consider the conditional probability
\[
P_1:=P\Big(\cdot \,\Big|\,  G( S_0 \cup \xi')|\cC=(\bar S, \bar E)\Big)\,,\]
where the above conditioning event $\{G( S_0 \cup \xi')|\cC=(\bar S, \bar E)\}$  is compatible with $E_1\cap E_2$ (considering a  regular version of the conditional probability $P(\cdot|\cT)$).
Under this conditioning event, the seed $\cS_1$ is determined (we choose  $\cS_1$ in a measurable way) and we denote by $S_1$ its value. A crucial observation is that the process $\xi'\setminus\cC$ under $P_{1}$ is a PPP on $\bbR^d\times \bbM$  with intensity  measure $\l \bar\varphi(t, \bar S) \mu(\md t)$ where $\mu$ denotes the restriction to $\L$ of $\cL\otimes \nu$.

\begin{figure}
\centering
\begin{tikzpicture}[scale=1.1]
\draw (-6,-6) node[above right]{$\L_{6N}$} rectangle (6,6);
\begin{scope}[shift={(-25/4,-0.5)}]
    \fill (0,0) node[left]{$z_0$}circle (2pt);
    \draw (-1,-1) rectangle (1,1);
    \draw[pattern=north east lines] (0.95,-0.05) rectangle (1.05,1.05);
    \fill (0.95,1.05) node[left]{$z_1$}circle (2pt);
\begin{scope}[shift={(0.95,1.05)}]
    \draw (-1,-1) rectangle (1,1) node[above]{$z_1+\Lambda_N$};
    \draw[pattern=north east lines] (0.95,-0.05) rectangle (1.05,-1.05);
    \fill (0.97,-0.8) node[left]{$z_2$}circle (2pt);
\begin{scope}[shift={(0.97,-0.8)}]
    \draw (-1,-1) rectangle (1,1);
    \draw[pattern=north east lines] (0.95,-0.05) rectangle (1.05,1.05);
    \fill (0.95,-0.05) node[left]{$z_3$}circle (2pt);
\begin{scope}[shift={(0.95,-0.05)}]
    \draw (-1,-1) rectangle (1,1);
    \draw[pattern=north east lines] (0.95,-0.05) rectangle (1.05,1.05);
    \fill (1.02,0.5) node[left]{$z_4$}circle (2pt);
\begin{scope}[shift={(1.02,0.5)}]
    \draw (-1,-1) rectangle (1,1);
    \draw[pattern=north east lines] (0.95,0.05) rectangle (1.05,-1.05);
    \fill (0.95,-0.3) node[left]{$z_5$}circle (2pt);
\begin{scope}[shift={(0.95,-0.3)}]
    \draw (-1,-1) rectangle (1,1);
    \draw[pattern=north east lines] (0.95,-0.05) rectangle (1.05,1.05);
    \fill (0.95,0.98) node[left]{$z_6$}circle (2pt);
\end{scope}
\end{scope}
\end{scope}
\end{scope}
\end{scope}
\end{scope}
\draw[dashed] (-7.25,0) -- (6,0);
\draw (-2,-2) rectangle (2,2) node[right]{$\Lambda_{2N}$};
\fill[opacity=0.3] (-0.7,-4.2) rectangle (0.2,4.2);
\fill (1.6,-1.3) node[left]{$z_{4d+2}$} circle (2pt);

\begin{scope}[shift={(1.3,-1.3)}]
    \draw (-1,-1) rectangle (1,1);
    \draw[pattern=north east lines] (0.95,-0.05) rectangle (1.05,1.05);
    \fill (0.95,-0.05) node[right]{$z_{4d+3}$}circle (2pt);
\begin{scope}[shift={(0.95,-0.05)}]
    \draw (-1,-1) rectangle (1,1);
    \draw[pattern=north east lines] (0.95,-0.05) rectangle (1.05,1.05);
    \fill (1.02,0.5) node[left]{$z_{4d+4}$}circle (2pt);
\begin{scope}[shift={(1.02,0.5)}]
    \draw (-1,-1) rectangle (1,1);
    \draw[pattern=north east lines] (0.95,-0.05) rectangle (1.05,1.05);
    \fill (0.95,0.3) node[left]{$z_{4d+5}$}circle (2pt);
\begin{scope}[shift={(0.95,0.3)}]
    \draw (-1,-1) rectangle (1,1);
    \draw[pattern=north east lines] (0.95,-0.05) rectangle (1.05,1.05);
    \fill (0.95,0.98) node[left]{$z_{4d+6}$}circle (2pt);
\begin{scope}[shift={(0.95,0.98)}]
\draw (-1,-1) rectangle (1,1);
\end{scope}
\end{scope}
\end{scope}
\end{scope}
\end{scope}

\draw (-4,-6) rectangle (4,-11/2);
\draw (11/2,-4) rectangle (6,4);
\draw (4,6) rectangle (-4,11/2);
\draw (2,-11/2) node[above]{$\Lambda^{(1)}$};
\draw (11/2,2) node[left]{$\Lambda^{(2)}$};
\draw (2,11/2) node[below]{$\Lambda^{(3)}$};
\end{tikzpicture}
\caption{\label{fig:sequence}Illustration of the proof of Lemma \ref{passo0}. The points $z_0,z_1,\dots$ are roughly where the successive seeds $S_0,S_1,\dots$ are located. The next one, $z_{i+1}$, is in the hatched box around the quarter-face $z_i+F_N^{1,\sigma^i}$ for suitable $\sigma^i$, modulo a $\sqrt{N}$ error. The dashed mid-line is a guide for the eye: when $z_i$ is below it, the upper quarter-face is chosen, when $z_i$ is above, the lower one is chosen. Notice that, due to the $\sqrt N$ errors made at each step, the position of $z_i$ may drift a little, in particular in the dimensions not shown in the figure. After six steps, we arrive in the gray region given by $[-N/2-6\sqrt N,6\sqrt N]\times[-4N-6\sqrt N,4N+6\sqrt N]^{d-1}$. We then use another $4(d-1)$ more steps (not displayed in the figure)  to correct the possible offset of order $\sqrt N$ in each of the dimensions, ending up with $z_{4d+2}$ inside $\Lambda_{2N}$. Finally, we proceed similarly from $z_{4d+2}$ to reach each of $\Lambda^{(1)}$, $\Lambda^{(2)}$, $\Lambda^{(3)}$. As visible in the figure, some care is needed for the last step in cases when e.g.\ $z_{4d+6}+\Lambda_N\not\subset\Lambda_{6N}$.}
\end{figure}

The construction presented next is illustrated in Figure~\ref{fig:sequence}.

 Fix $u>0$ to be chosen later. We apply Lemma~\ref{lemma_udine} with $R:=\Lambda_N(z_1)\times\bbM$, $A:=\bar S
 $, $B:=(z_1+F^{1,\s^1}_N)\times \bbM$ with $\s^1$ the sequence of the signs of the coordinates of $-z_1$ (exceptionally, in this section we define  the sign of $0$ as $+1$),  $\tilde\bbX:=\xi_1 \cap R$ under $P_1$ (and therefore under $P$), $\bbY:=(\xi\setminus\bar S) \cap R$ under $P_1$,  $\bbX:=\xi \cap R $  under $P$. By Lemmas~\ref{lemma_udine}  and \ref{lemma_crescita} 
we then get 
\begin{multline}\label{suono2}
P_1\left(\bar S\leftrightarrow(z_1+F^{1,\s^1}_N)\times\bbM\text{ in }G\left(\bar S, \xi_1\cap R\right)\cup G(((\xi\setminus\bar S)\cup\xi_1)\cap R)\right)\\
\begin{aligned}[t]
&{}\ge  \bbP\left(A \leftrightarrow(z_1+F^{1,\s^1}_N)\times\bbM\text{ in }G\left(A, \tilde\bbX\right)\cup G\left(\tilde\bbX\cup\bbY\right)\right)\\
&{}\ge \left(1-e^{-(\lambda'-\lambda) u/(400d)}\right)\left(1-e^{\lambda u} P\left(A\not\leftrightarrow (z_1+F^{1,\s^1}_N)\times\bbM 
\text{ in } G\left(\left(R\cap\xi\right)\cup A\right)\right)\right)\\
&{}\ge \left(1-e^{-(\lambda'-\lambda) u/(400d)}\right)\left(1-e^{\lambda u-c/\varepsilon}\right)=:\a_2.\end{aligned}\end{multline}
We take $u:=400d/((\lambda'-\lambda)\sqrt \varepsilon)$ and using that $\varepsilon$ is small enough depending on $\lambda',\lambda,d,c$ we can make $\a_2$ and $\a_1$ from \eqref{suono1} arbitrarily close to $1$.

Let  $E_3$ be the event that $\cC\leftrightarrow(z_1+F^{1,\s^1}_N)\times\bbM\text{ in }G\big(\cC, \xi_1\cap R\big)\cup G\big(((\xi\setminus\cC)\cup\xi_1)\cap R\big)$. By combining \eqref{suono1} and \eqref{suono2}, we get $ P(E_1\cap E_2\cap E_3) \geq \a_1\a_2$. 
Since $P(\cD^c)\leq \e$,  we have 
\[P(E_1\cap E_2\cap E_3\cap \cD) \geq \a_1\a_2-\e\,.\]
We call $\cC'$ the set of points $t\in S_0\cup\big((\xi\cup\xi_1)\cap (\Lambda\times\bbM)\big)$ such that $\cC\leftrightarrow t $ in the graph $\G:=G\left(S_0\cup\big( (\xi\cup\xi_1)\cap(\Lambda\times\bbM)\big)\right)$. The event $E_1\cap E_2\cap E_3 \cap \cD$ and the decreasing property of $\cH_*$ in Lemma \ref{lem:seed:to:seed}\ref{cond:decreasing} imply that the graph  $\G|(\cC'\cap(\Lambda_{6N}\times\bbM))$  belongs to $\cH_*$. Furthermore,  $E_1\cap E_2\cap E_3 \cap \cD$   implies that there is a path of length at least $\sqrt{N}-1$ inside $\Gamma|\cC'$ intersecting 
$z_1+F^{1,\s^1}_N$. Hence, by Lemma \ref{lem:seed:to:seed}\ref{cond:deterministic}, $E_1\cap E_2\cap E_3 \cap \cD$  implies the following event $E_4$: $\cC'$ contains a seed 
  $\cS_2\subset  \L_{n-K}(z_2)\times\bbM_g\subset (z_1+\tilde F_N^{1,\s^1}+\L_{\sqrt{N}})\times\bbM_g$. Note that by the above observations 
\be\label{suono3}
P( E_1\cap E_2\cap E_3\cap E_4)\geq \a_1\a_2-\e=\a_3\,.
\en

We repeat the same reasoning 4 more times (using a new $\xi_i$ each time) until we obtain a seed $\cS_6\subset \Lambda_{n-K}(z_6)\times\bbM_g\subset [-N/2-6\sqrt N,6\sqrt N]\times [-4N-6\sqrt N,4N+6\sqrt N]^{d-1}\times\bbM_g$. Note that the projection on the first coordinate of $\Lambda_{n-K}(z_6)$ is included in $[-2N,2N]$ (by taking $N$ large). We then seek to adjust the other coordinates  by repeating the same procedure another $4(d-1)$ times.  We start with the second coordinate. For $i\in\{6,7,8,9\}$ we try to connect $S_{i}$ with $z_i+F_N^{2,\s^i}$ where  $\sigma^i$ is the sign sequence of $-z_i$, thus leading to the seed $\cS_{i+1}\subset \L_{n-K}(z_{i+1})\times\bbM_g$. Then we do the same for the third coordinate and so on. Notice that, at each step when adjusting the $j$--th coordinate, $|z_{i+1}(j)|\le \max(|z_i(j)|-3N/4,3N/2)$, while $|z_{i+1}(j')|\le \max(|z_i(j')|+\sqrt N,3N/2)$ for $j'\in\{1,\dots,d\}\setminus\{j\}$.

Recalling that $N$ is taken large enough, this yields a seed
\[\cS_{4d+2
}\subset\Lambda_{n-K}(z_{4d+2
})\times\bbM_g\subset\Lambda_{2N}\times\bbM_g.\]

From this point, for each $i\in\{1,2,3\}$ we proceed similarly going from $\cS_{4d+2
}$ towards $\Lambda^{(i)}\times\bbM$. For concreteness, consider $i=2$, the others being analogous. Further for concreteness, assume that $z_{4d+2
}\in[N-\sqrt N,2N]\times [-2N-3\sqrt N,2N+3\sqrt N]^{d-1}$ (if this is not the case, up to three additional steps are performed to reach this case, setting $\sigma_1=1$ rather than the sign of the first coordinate of $-z_j$). We perform another 3 steps to the right, obtaining $z_{4d+5
}\in[4N-4\sqrt N,5N+3\sqrt N]\times[-2N-6\sqrt N,2N+6\sqrt N]^{d-1}$. We now consider three cases for the first coordinate $\zeta$ of $z_{4d+5
}$.

If $\zeta\in[9N/2+2\sqrt N,5N-2\sqrt N]$, then an additional step as before yields the desired seed $\cS_{4d+6
}\subset\Lambda^{(2)}$ and we are done. Assume that $\zeta\in[4N-4\sqrt N,9N/2+2\sqrt N]$ and perform a step as before to get that the first coordinate of $z_{4d+6
}$ belongs to $[5N-5\sqrt N,11N/2+3\sqrt N]$. At this point, we need more care, since $\Lambda_N(z_{4d+6
})\not\subset\Lambda$. However, we may proceed in the same way as above, setting $R=\Lambda_N(z_{4d+6
})\cap((-\infty,23N/4]\times\bbR^{d-1})\subset \Lambda$ and replacing $z_{4d+6
}+F^{1,\sigma^{4d+6}}_N$ by $F'=R\cap([23N/4-1,23N/4]\times \bbR^{d-1})$. Indeed, for any seed $S\subset \Lambda_{n-K}$ and $x\in[n+1,N]$, we have
\begin{multline*}
P(S\leftrightarrow [x-1,x]\times [-N,N]^{d-1}\times\bbM\text{ in }G((S\cup\xi)\cap(\Lambda_N\times\bbM)))\\
\ge P(S\leftrightarrow F_N\times\bbM\text{ in }G((S\cup\xi)\cap (\Lambda_N\times\bbM)))\ge 1-e^{c/\varepsilon},\end{multline*}
using Lemma~\ref{lemma_crescita}. Noticing that $F'+\Lambda_{\sqrt N}\subset \Lambda^{(2)}$, we obtain a seed $\cS_{4d+7
}\subset \Lambda^{(2)}$ as desired. Finally, the remaining case $\zeta\in[5N-2\sqrt N,5N+3\sqrt N]$ is treated like $\zeta\in[4N-4\sqrt N,9N/2+2\sqrt N]$ omitting the first step, i.e.~the step used there to suitably localize  the first coordinate of $z_{4d+6
}$.

In total, in at most $4d+26
$ steps, we have obtained that
\[P((0,0)\text{ is occupied})\ge f^{\circ (4d+25
)}(\alpha_1)\,,\]
where $f(x)=\alpha_2x-\varepsilon$ for $x\in\bbR$ and $f^{\circ m}$ denotes the $m$-fold composition of the function $f$. Recalling \eqref{suono1} and \eqref{suono2}, by our choice of scales \eqref{eq:scales}, this completes the proof of Lemma~\ref{passo0}.
\end{proof}

Recall the notation of Definition~\ref{def:occupied}. From now on we choose the seeds $\cS^{(1)},\cS^{(2)},\cS^{(3)}$ appearing in Definition~\ref{def:occupied} in a measurable way in terms of  the connected components intersecting $S_0$ in 
    \[G\Big(S_0\cup \big((\xi\cup\xi_{(0,100d]}) \cap(\Lambda\times\bbM)\big)\Big).\]
Linking events are defined very similarly to occupation, as follows (see Figure~\ref{fig:renorm}).
\begin{Definition}[Link]
\label{def:link}
    Assume that $(0,0)$ is occupied. 
    Let $\L'=\L\cup \L_{6N}(12N e_1)$ and set 
    \[\Gamma:=G\Big(S_0\cup\big((\xi\cup\xi_{(0,100d]})\cap(\Lambda'\times\bbM)\big)\cup\big(\xi_{(100d,200d]}\cap(\Lambda_{6N}\times\bbM)\big)\Big)\]
    We denote $\Lambda_2^{(i)}=12Ne_1+\Lambda^{(i)}$ for $i\in\{1,2,3\}$. We say that $(1,0)$ is \emph{linked} to $(0,0)$, if there exists a seed    \begin{align*}\cS^{(i)}_2&{}\subset\{t\in\Lambda^{(i)}_2\times\bbM_g:t\leftrightarrow\cS^{(2)}\text{ in }\Gamma\},
\end{align*}
    for each $i\in\{1,2,3\}$.
\end{Definition}
More generally, we can define the link variables simultaneously with the exploration process of Section~\ref{subsec:exploration}. Whenever a link variable is about to be explored in \eqref{reveal}, we define it as in Definition~\ref{def:link}. The only modifications are that $\Lambda^{(i)}$ may need to be rotated depending on the entry vertex of the link and that the ranges of the sprinkling PPP $\xi_{(0,100d]}$ and $\xi_{(100d,200d]}$ may need to be modified. More precisely, assume that we are currently exploring a link to construct the set $C^s_{j+1}$ (recall Section~\ref{subsec:exploration}) with $s\in\{0,\dots,M-1\}$ and $j\in\{2,\dots,\sharp\Lambda'_L-1\}$ (the cases $j=0$ and $j=1$ are exactly given by Definitions~\ref{def:occupied} and~\ref{def:link}, up to translation) and assume that a point $x_{j+1}^s\in\Delta E^s_{j}$ is to be added to $E_{j}^s\cup F_{j}^s$, as otherwise there is nothing to do. For each $x\in \cX:=(E^s_{j}\setminus\{(0,s)\})\cup F^s_j\cup\{x^s_{j+1}\}$, let $i_x\in\{1,2,3,4\}$ be the number of explored links with endpoint $x$, including the link with endpoint $x^s_{j+1}$ currently being explored. Then the vertex set of the graph $\Gamma$ in the analogue of Definition~\ref{def:link} is
\begin{multline*}
\big(S_0+12Nse_2\big)\cup\Big(\big(\xi\cup\xi_{(0,100d]}\big)\cap\big(\Lambda_{6N}(-12Ne_1+12Nse_2)\times\bbM\big)\Big)\\
\Big(\big(\xi\cup\xi_{(0,200d]}\big)\cap\big(\Lambda_{6N}(12Nse_2)\times\bbM\big)\Big)\cup\bigcup_{x\in\cX}\Big(\big(\xi\cup\xi_{(0,100di_x]}\big)\cap\big(\Lambda_{6N}(\bar x)\times\bbM\big)\Big).\end{multline*}

\begin{Lemma}[Linking probability]
\label{lem:domination}
    The link and occupation events defined above satisfy that $P$ dominates $p_0$.
\end{Lemma}
The proof is omitted, as it is identical to the one of Lemma~\ref{passo0}, up to heavier notation.

We are now ready to complete the proof of Theorem~\ref{teo_crossings} by proving \eqref{2d_bound}.
\begin{proof}[Proof of \eqref{2d_bound}]
    By Corollary~\ref{cor:domination} and Lemma~\ref{lem:domination}, for some constants $\k,\k'>0$ and all positive $M,L$, we have $P(N_L\ge \k M)\ge 1-e^{-\k'M}$. By the reverse Fatou's lemma, this gives that 
    \[P(N_L\ge \k M\text{ for infinitely many $L$})\ge 1-e^{-\k'M}.\]
    But the above event implies that there are at least $\k M$ vertex-disjoint paths from occupied vertices in $\{0\}\times\{0,\dots,M-1\}$ to infinity following links in \[\Lambda_\infty':=\bigcup_{L\ge 1}\Lambda'_L=\{0,\dots,M-1\}^2\cup\{(x,y)\in\bbZ^2:x\ge M\}.\] By Definitions~\ref{def:occupied} and~\ref{def:link}, each such renormalised path corresponds to a path in $G(S_0\cup\xi^*)$ that starts at (a vertical translate) of the seed $S_0$ which is not necessarily present in $\xi^*$, but which is at distance at least $1$ from $\Lambda_{6N}$. Since all other vertices are in $\xi^*$ and bonds are present in $G(\xi^*)$, we obtain infinite vertex-disjoint paths starting to the left of $\Lambda_{\ell}$. Moreover, since, in $\Lambda_\infty'$, the only way of going to infinity requires intersecting $\{M\}\times\{0,\dots,M-1\}$, the above paths give crossings of $\Lambda_{\ell}$, completing the proof of \eqref{2d_bound}.
\end{proof}

\section*{Acknowledgements}
The authors thank Lorenzo Dello Schiavo, Barbara Dembin, Markus Heydenreich and Mathew Penrose for helpful discussions and the anonymous referees for the careful proofreading.

\paragraph{Declaration on the use of AI} No artificial intelligence tools were used in the development of the mathematical ideas, arguments, or proofs presented in this paper, nor were they used for drafting, writing or polishing the manuscript.

\appendix 

\section{The  effective homogenized matrix \texorpdfstring{$D(\rho)$}{D(rho)}}
\label{app_fine}
Suppose $(\l,\nu,\varphi)$ satisfies Assumption~\ref{assumere} and take $\rho\geq\l$.
In this appendix we explain 
how to fit  the infinite percolation cluster  with unit conductances of our RCM  to the general setting of \cite{Faggionato25}.
To this aim it is necessary to leave the probability space $\O\times [0,1]^J$ with probability $P_{\rho,\nu}\otimes Q$ and to deal with the space $\cN(\bbW)$  endowed with the probability measure $\cP_\rho$ obtained by pushing forward $P_{\rho,\nu}\otimes Q$  by $\Psi$ (recall Section~\ref{sec_FKG}).

To each $\z\in\cN(\bbW)$ we associate a family of edges in $\bbR^d$ denoted by $E(\z)$, given   by the pairs $\{x,y\}$ such that  for some $(\{t,t'\}, u)$ it holds $x=\pi(t)$, $y=\pi(t')$ and $u\leq \varphi(t,t')$.  If the graph with the above edge set $E(\z)$ (and vertices their endpoints) has a unique infinite cluster, we denote by $\cC(\z)$ this cluster, otherwise we set $\cC(\z):=\varnothing
$. Finally, we define $\cE(\z)$ as the family of edges $\{x,y\}\in E(\z)$ with vertices $x,y$ in $\cC(\z)$.

We introduce the simple point process (SPP) $\cN(\bbW)\ni\z\to \hat \z \in \cN(\bbR^d)$ defined as $\hat \z:= \cC(\z)$ and the conductance field  $c(x,y,\z)$ defined as $c(x,y,\z):=\1( \{x,y\}\in E(\z))$. Then the weighted graph in \cite{Faggionato25} associated to the above SPP and conductance field coincides in law to the infinite cluster of the RCM$(\rho,\nu,\varphi)$.

The Abelian group $\bbG=\bbR^d$  acts on the Euclidean space $\bbR^d$
 by the translations $(\t_g)_{g\in\bbG}$, where $\t_g x:=x+g$ for $x\in\bbR^d$.
 $\bbG$ acts on  on $\cN(\bbW)$  by  the action $(\theta_g)_{g\in\bbG}$ where, given $\z\in\cN(\bbW)$, the set $\theta_g \z $ is given by the elements 
$\left(\{(x-g,m),(x'-g,m')\},u\right)$ as $\left(\{( x,m),(x',m')\},u\right)$ varies in $\zeta$. Note that $\cP_\rho$ is stationary and ergodic for the action $(\theta_g)_{g\in\bbG}$. Indeed $\cP_\rho$, is even mixing (the proof of mixing is simplified by Assumption~\ref{assumere}\ref{ass:finite:support}).

At this point, by Assumption~\ref{assumere}, for $\rho\geq \l$ we are in the setting fully satisfying the assumptions of   \cite{Faggionato25}.
In particular, 
the effective homogenized matrix $D(\rho)$ is given by the following definition where 
 $\cP^{(0)}_\rho$ is the Palm distribution associated to $\cP_\rho$ and the SPP $\zeta\mapsto\cC(\zeta)$ (roughly $\cP^{(0)}_\rho:=\cP_\rho(\cdot|0\in\cC(\z))\,$):
\begin{Definition}[Effective homogenized matrix $D(\rho)$] \label{def_D} Suppose that the triple $(\l,\nu,\varphi)$ satisfies Assumptions~\ref{assumere} and take $\rho\geq \l$. 
Then  the \emph{effective homogenized matrix} $D(\rho)$  is defined as the unique symmetric $d\times d$--matrix such that, for each $a\in\bbR^d$,
\be\label{eq_D}
 a\cdot D(\rho)a=\inf_{f\in L^\infty(\cP^{(0)}_\rho)}\frac{1}{2}\int\cP^{(0)}_\rho(d\zeta)\sum_{x\in\cC(\z):\{0,x\}\in \cE(\z)} \left(a\cdot x-\nabla f (\zeta,x)\right)^2\,,
\en
where   $\cP^{(0)}_\rho$ is the Palm distribution associated to $\cP_\rho$ and the SPP $\zeta\mapsto\cC(\zeta)$ (roughly $\cP^{(0)}_\rho:=\cP_\rho(\cdot|0\in\cC(\z))$) and  $ \nabla f (\zeta, x) =f(\theta_x \zeta)-f(\zeta)$.
\end{Definition}
It is simple to check that $D$ is proportional to the identity.
Indeed, by the above definition and Assumption~\ref{assumere}\ref{ass:symmetry}, we get that $D(\rho)$ has to be left invariant by permutations of coordinates and that $D \Phi_i=\Phi_i D$ where $\Phi_i $ is the linear operator on $\bbR^d$ flipping the sign of the $i$-th entry. It is then simple to conclude that $D$ is a scalar matrix.

Due to \eqref{eq_diagonale} any unit vector $e$ is an eigenvector of $D(\rho)$. In particular, Proposition~\ref{prop_scaling} is now an immediate consequence of \cite{Faggionato25}*{Corollary~2.7}.

\section{Transport via diffusions on singular geometric structures}
\label{transport_diff}
We conclude by discussing  another application of Theorem~\ref{teo_crossings} to transport, where now modelling is based on  
 diffusions on random singular geometric structures (see e.g.~\cite{Zhikov06}).
 
Suppose that the triple $(\l,\nu,\varphi)$ satisfies  Assumptions~\ref{assumere}.  We consider the RCM$(\rho,\nu,\varphi)$ where  $\rho\geq  \l$. We then  consider the 1d diffusion (with unit diffusion coefficient) on the graph $\bar{\cC}(\o,\underline u) $ given by the infinite cluster $\cC(\o,\underline{u})$ thought of as immersed in $\bbR^d$ with edges given by line segments (we take the segment $\overline{xy}$ from $x$ to $y$ instead of the pair $\{x,y\}$).  
We introduce the random measure $\mu(\o, \underline u ):= \sum _{ \{x,y\} }  \md s_{x,y}$, 
where  the sum is over the edges of the infinite cluster and $\md s_{x,y}$  is the  arc-length measure (i.e.~the one-dimensional Hausdorff measure $\cH^1$) on the segment $\overline{xy}$ from $x$ to $y$. We point out that in this section  we use the term diffusion without referring to a particular  stochastic process, but to the generalized differential operator $v\mapsto -{\rm div} \nabla_{\mu(\o,\underline u)} v$. We refer to \cite{Zhikov06} for its definition  (which anyway will not be used  in the rest). One can study stochastic homogenization for this model with the approach of \cite{Zhikov06}*{Section 7}, where  the same problem  is considered for  the supercritical cluster of Bernoulli bond percolation. 

Following step-by-step \cite{Zhikov06} we consider the differential equation    $  -{\rm div} \nabla_{\mu(\o,\underline u)} u_\ell =0$ on the box $\L_\ell$ 
  with  the following boundary conditions: $u_\ell\equiv 0$ on the left face of $\L_\ell$  
along the first direction, $u_\ell\equiv \ell$ on the right face, while $u_\ell$ satisfies Neumann boundary conditions on the other faces. This equation corresponds to \cite{Zhikov06}*{Eq.~(7.2)}, after rescaling space to map $\L_\ell$ to the unit box $\L_{1/2}$ and permuting the coordinates. Moreover, the  equation should be understood in a weak sense as for \cite{Zhikov06}*{Eq.~(7.2)}.  In particular, $u_\ell$ must  be a  continuous function  on the support of $\mu_{|\L_{\ell}}$ such that $u$ has  $\mu$--square integrable weak gradient along the arcs of $\bar{\cC}(\o,\underline u)$.
    As done in \cite{Zhikov06}*{Section~7} for \cite{Zhikov06}*{Eq.~(7.2)} and the Bernoulli supercritical cluster,  one can apply the  general theory developed in \cite{Zhikov06}, in particular \cite{Zhikov06}*{Theorem~6.3} and the discussion following it  concerning the convergence of energies. One then obtains the stochastic homogenization  for $u_\ell$ with effective homogenized  matrix of the form $\k_{\rm eff} \mathbb{I}$, where the deterministic constant $\k_{\rm eff} $ satisfies a.s.\ (see \cite{Zhikov06}*{Eq.~(6.19)})
    \[
    \k_{\rm eff}= \lim _{\ell\to+\infty}\frac{1}{(2{\ell})^d} \int _{\L_{\ell}} |\nabla u_\ell |_2^2 d\mu\,, 
    \]
    at the condition that $\k_{\rm eff}$ itself is strictly positive. In analogy with electric theory, it is natural to call $\k_{\rm eff}$ effective  conductivity  along the first direction.

\begin{Definition}[Property $P(C)$]
\label{def_PC}
Given $C>0$ we say that property $P(C)$ holds for $(\o,\underline{u})$, if,  for ${\ell}$ large enough, the following holds. If  $u$ 
is  a  continuous function  on the support of $\mu_{|\L_{\ell}}$ such that $u$ has  $\mu$--square integrable weak gradient along the arcs of $\bar{\cC}(\o,\underline u)$; $u$  equals zero on $\{x\in\L_{\ell}: x_1=-{\ell}\}$ and $u$ equals $2\ell$ on $\{x\in\L_{\ell}: x_1={\ell}\}$, then   it holds
\be \label{stima}
\frac{1}{(2{\ell})^d} \int _{\L_{\ell}} |\nabla u |_2^2 d\mu \geq C \,.
\en
\end{Definition}

By the previous discussion, if for some $C>0$ property $P(C)$ holds $P_{\rho,\nu}$-a.s.,  then $\k_{\rm eff}>0$.
Note that Theorem~\ref{teo_crossings} combined with Remark~\ref{docile} corresponds to \cite{Zhikov06}*{Theorem~7.1} for the Bernoulli bond percolation on $\bbZ^d$. Hence, by adapting the  arguments used in \cite{Zhikov06}*{Section~7} to our setting, we get 
the following.
\begin{Proposition}
    \label{teo3}
Suppose that the triple $(\l,\nu,\varphi)$ satisfies  Assumptions~\ref{assumere}. Then for any $\rho>\l$  there exists $C>0$ such that property $P(C)$  holds $P_{\rho,\nu}$-a.s.\ and in particular $ \k_{\rm eff} >0$.
\end{Proposition}
\begin{proof}  
Suppose that $y_0,y_1,\dots, y_n$ is a LR crossing of $\L_{\ell}$ included in the infinite cluster.
Calling $\s$ the $\cH^1$--measure with support $\L_{\ell}\cap\left(\bigcup _{i=0}^{n-1} \overline{y_i y_{i+1}}\right)$,  by applying the Cauchy--Schwarz inequality we get 
$
2{\ell}= \int \nabla u \md\s \leq L^{\frac{1}{2}}( \int |\nabla u|_2^2 \md\s )^\frac{1}{2}  $, where $L$ is the total mass of $\s$.
Note that $L\leq n{\ell}_*$ by Assumption~\ref{assumere}\ref{ass:finite:support}. Having obtained that   $(2{\ell})^{-d}\int |\nabla u|^2 \md\s\geq  (2\ell)^{2-d}/(n {\ell}_*)$ for any LR crossing of $\L_{\ell}$, and summing the contribution of a maximal set of vertex-disjoint LR crossings with cardinality $\cN_{\ell}$, we then get 
that the l.h.s.\ of \eqref{stima} is lower bounded by $(2\ell)^{2-d}{\ell}_*^{-1}\sum_{i=1}^{\cN_{\ell}} \frac{1}{n_i }$, where the $i$-th LR crossing has $n_i+1$ vertices. We conclude the proof by applying Jensen's inequality, as in Lemma~\ref{lemma_LB}, and Theorem~\ref{teo_crossings}, as done for resistor networks, or by using  directly Remark~\ref{docile}.
\end{proof}

\bibliographystyle{plain}
\bibliography{Bib}

\end{document}